\documentclass[11pt]{article}

\usepackage{amsmath}
\usepackage{amsfonts}
\usepackage{graphicx}
\usepackage{setspace}
\usepackage{amsmath}
\usepackage{amssymb}
\usepackage{latexsym}
\usepackage{amsmath, amsfonts,amssymb, amsthm, euscript,makeidx,color,mathrsfs}

\usepackage[numbers,sort&compress]{natbib}

\def\sqr#1#2{{\vcenter{\vbox{\hrule height.#2pt
              \hbox{\vrule width.#2pt height#1pt \kern#1pt \vrule width.#2pt}
              \hrule height.#2pt}}}}
\usepackage{hyperref}
\hypersetup{
    colorlinks=true,
    linkcolor=blue,
    urlcolor=cyan,
    citecolor=cyan
}
\usepackage{caption}

\RequirePackage[capitalize,nameinlink]{cleveref}

\crefname{section}{section}{sections}
\crefname{subsection}{subsection}{subsections}
\Crefname{section}{Section}{Sections}
\Crefname{subsection}{Subsection}{Subsections}

\crefname{condition}{Condition}{Conditions}

\Crefname{figure}{Figure}{Figures}

\crefformat{equation}{\textup{#2(#1)#3}}
\crefrangeformat{equation}{\textup{#3(#1)#4--#5(#2)#6}}
\crefmultiformat{equation}{\textup{#2(#1)#3}}{ and \textup{#2(#1)#3}}
{, \textup{#2(#1)#3}}{ and \textup{#2(#1)#3}}
\crefrangemultiformat{equation}{\textup{#3(#1)#4--#5(#2)#6}}%
{ and \textup{#3(#1)#4--#5(#2)#6}}{, \textup{#3(#1)#4--#5(#2)#6}}{ and \textup{#3(#1)#4--#5(#2)#6}}

\Crefformat{equation}{#2Equation~\textup{(#1)}#3}
\Crefrangeformat{equation}{Equations~\textup{#3(#1)#4--#5(#2)#6}}
\Crefmultiformat{equation}{Equations~\textup{#2(#1)#3}}{ and \textup{#2(#1)#3}}
{, \textup{#2(#1)#3}}{ and \textup{#2(#1)#3}}
\Crefrangemultiformat{equation}{Equations~\textup{#3(#1)#4--#5(#2)#6}}%
{ and \textup{#3(#1)#4--#5(#2)#6}}{, \textup{#3(#1)#4--#5(#2)#6}}{ and \textup{#3(#1)#4--#5(#2)#6}}

\crefdefaultlabelformat{#2\textup{#1}#3}

\newtheorem {theorem}{Theorem}[section]
\newtheorem {lemma}[theorem]{{\bf Lemma}}
\newtheorem {corollary}[theorem]{{\bf Corollary}}
\newtheorem {proposition}[theorem]{{\bf Proposition}}
\theoremstyle{remark}
\newtheorem {remark}{{\bf Remark}}[section]

\theoremstyle{definition}

\theoremstyle{plain} \numberwithin {equation}{section}

\numberwithin{assumption}{section}

\allowdisplaybreaks

\usepackage{graphicx}
\usepackage{ifpdf}
\ifpdf
  \usepackage{epstopdf}
\fi

\usepackage{enumerate}

\def\deq{\mathop{\buildrel\Delta\over=}}

\begin{document}

\title{Null controllability for semi-discrete stochastic semilinear parabolic equations\footnote{
    This work is supported by the NSF of China under
grant 12401589, the Fundamental Research Funds for the Central Universities 2682024CX013, and the Key Project of Sichuan Science and Technology \& Education Joint Foundation under grant of 2024NSFSC1963.
}}
\author{
	Yu Wang\footnote{School of Mathematics, Southwest Jiaotong University, Chengdu, P. R. China.
	Email: yuwangmath@163.com.}
	~~~ and ~~~
	Qingmei Zhao\footnote{Corresponding author. School of Mathematics Sciences, Sichuan Normal University, Chengdu, P. R. China. E-mail:
	qmmath@163.com.}
} 

\date{}

\maketitle

\begin{abstract}
The global null controllability of  stochastic semilinear parabolic equations with globally Lipschitz nonlinearities  has been addressed in recent literature. 
However, there are no results concerning their numerical approximation and the behavior of discrete controls when the discretization 
parameter goes to zero. This paper is intended to studying the null controllability for
semi-discrete stochastic semilinear parabolic equations, where the spatial variable is discretized with finite difference scheme and the time is kept as a continuous variable. The proof is based on a new refined semi-discrete Carleman estimate and Banach fixed point method. The main novelty
here is that the Carleman parameters and discretization 
parameter are made explicit and are then used in a Banach fixed point
method. 
\end{abstract}
\bigskip

\bigskip


\section{Introduction}

Let $T>0$ and $(\Omega, \mathcal{F}, \mathbf{F},
\mathbb{P})$ with $\mathbf F=\{\mathcal{F}_{t}\}_{t \geq 0}$ be a complete filtered probability space on which a one-dimensional standard Brownian motion $\{W(t)\}_{t \geq 0}$ is defined and $\mathbf{F}$ is the natural filtration generated by $W(\cdot)$,  augmented by all the $\mathbb{P}$ null sets in $\mathcal{F}$. 
Write $ \mathbb{F} $ for the progressive $\sigma$-field with respect to $\mathbf{F}$. 
Let $\mathcal{H}$ be a Banach space. Denote by $L^2_{\mathcal{F}_\tau}(\Omega; \mathcal{H})$ the space of all $\mathcal{F}_\tau$-measurable random variables $\xi$ such that $\mathbb{E} | \xi |^2_{\mathcal{H}}<\infty$; by $L^2_{\mathbb{F}}(0,T;\mathcal{H})$ the
Banach space consisting of all $\mathcal{H}$-valued
$ \mathbf{F} $-adapted processes $X(\cdot)$ such that
$\mathbb{E}(|X(\cdot)|^2_{L^2(0,T;\mathcal{H})})<\infty$, with the
canonical norm; by $L^\infty_{\mathbb{F}}(0,T;\mathcal{H})$ the
Banach space consisting of all $\mathcal{H}$-valued
$ \mathbf{F} $-adapted essentially bounded processes;
and by $L^2_{\mathbb{F}}(\Omega; C([0,T];\mathcal{H}))$ the Banach
space consisting of all $\mathcal{H}$-valued
$ \mathbf{F} $-adapted continuous processes
$X(\cdot)$ such that
$\mathbb{E}(|X(\cdot)|^2_{C([0,T];\mathcal{H})})<\infty$.

Let $ G = (0,1) $ and $ G_{0} $ be a nonempty open subset of $ G $.
As usual, $ \chi_{G_{0}} $   denotes the characteristic function of $ G_{0} $. We consider the following stochastic semilinear parabolic problem in a continuous framework:
\begin{align}
    \label{eqContinuousEq}
	\begin{cases}
		d y - y_{xx} d t = ( f(\omega, t, x,y) + \chi_{G_{0}} u ) d t +  ( g(\omega, t, x,y) +  v)  d W(t) & \text { in } (0,T) \times G, \\ 
		y(t,0) = y(t,1) = 0 & \text { on } (0,T), \\ 
		y(0) = y_{0} & \text { in } G 
		.
	\end{cases}
\end{align}
Here, the initial datum $ y_{0} \in  L^{2}(G) $ and the controls $ (u, v) \in L^{2}_{\mathbb{F}}(0,T; L^{2}(G_{0})) \times  L^{2}_{\mathbb{F}}(0,T; L^{2}(G)) $.
The null controllability problem of \cref{eqContinuousEq} consists of finding controls $ (u, v)$ such that $y(T)=0$ in $G$, $\mathbb{P} \text{-a.s.}$

As we known, the controllability of deterministic nonlinear parabolic equations has been studied by many authors and the results available in the literature are very rich (see e.g., \cite{Barbu2000,Balch2020,Doubova,Fernandez1997,Fernandez2000,Fursikov1996} and the references cited therein). 
A common feature of these results is that they usually use linearization and suitable fixed point methods to address the controllability problem of nonlinear systems. 
To this end, the important property of compactness is often required.

In recent years, the controllability of stochastic parabolic equations has received much attention. However, most of the existing results are focused only on stochastic linear parabolic equations, with very few results on the nonlinear case.
The main reason is that the compactness property, which is one of the key tools in the deterministic setting, is known to be false for the functional spaces related to stochastic setting.
To the best of our knowledge, the only known results in this direction are \cite{Hernandez2022,Hernandez2023,ZhangL,ZhangS}.
In \cite{Hernandez2023}, the authors prove the global null controllability of forward  stochastic semilinear parabolic equations with globally Lipschitz nonlinearities using a refined Carleman estimate and the Banach fixed point theorem in weighted Sobolev spaces. 
This method effectively circumvents the well-known problem of the lack of compactness embeddings for the solution spaces arising in the study of controllability problems for stochastic parabolic equations. 

Specifically, consider the following assumptions for $ f $ and $ g $:
\begin{enumerate}[(\text{A}$1$).]
    \item For each $ \varphi \in L^{2}(G) $, $ f(\cdot, \cdot, \cdot, \varphi) $ and $ g(\cdot, \cdot, \cdot, \varphi) $ are $ L^{2}(G) $-valued $ \mathbf{F}$-adapted processes;
    \item There exists $ L > 0 $, for $ (\omega, t, x, s_{1}, s_{2}) \in \Omega \times [0,T] \times G \times \mathbb{R} \times \mathbb{R} $, it holds that 
    \begin{align*}
        | f (\omega, t, x, s_{1}) - f(\omega, t, x, s_{2})| \leq L |s_{1} - s_{2}|;
    \end{align*}
    \item For $ (\omega, t, x) \in \Omega \times [0,T] \times G $, $ f(\omega, t, x, 0) = 0$;
    \item There exists $ L > 0 $ for $ (\omega, t, x, s_{1}, s_{2}) \in \Omega \times [0,T] \times G \times \mathbb{R} \times \mathbb{R} $, it holds that 
    \begin{align*}
        | g (\omega, t, x, s_{1}) - g(\omega, t, x, s_{2})| \leq L |s_{1} - s_{2}|;
    \end{align*}
    \item For $ (\omega, t, x) \in \Omega \times [0,T] \times G $, $ g(\omega, t, x, 0) = 0$.
\end{enumerate}
Assuming   (A1)--(A3),  then the system \cref{eqContinuousEq} is null controllable, i.e. for every $T>0$ and for every $y_0\in L^2(G)$, there exist controls $ (u, v) \in L^{2}_{\mathbb{F}}(0,T; L^{2}(G_{0})) \times  L^{2}_{\mathbb{F}}(0,T; L^{2}(G)) $ such that the unique solution $y$ of \cref{eqContinuousEq} satisfies $y(T, x ) = 0$ in $G$, $\mathbb{P}$-a.s. (see \cite[Theorem 1.1]{Hernandez2023}).

As the numerical approximation of the control system \cref{eqContinuousEq}, it is natural to consider the null controllability for
semi-discrete stochastic semilinear parabolic equations. 
We shall introduce a finite difference approximation of the
system \cref{eqContinuousEq}. 

Let $ N \in \mathbb{N} $.
Define the space discretization parameter $ h \deq 1/(N+1) $.
Consider the pairs $ (t, x_{i}) $ with $ t \in (0,T) $ and $ x_{i} = i h $ for $ i = 0, \cdots, N+1 $. 
Applying the centered finite difference method to the space variable for the system \cref{eqContinuousEq}, we consider the following semi-discrete stochastic semilinear system:
\begin{align}
	\label{eqDiscreteEq}
	\begin{cases}
		\begin{aligned}
        &d y^{i} - \frac{1}{h^{2}} ( y^{i+1} - 2 y^{i} +  y^{i-1} ) d t 
		\\ 
		&\quad \quad 
		= (f^{i}(y^{i}) + \chi^{i}_{G_{0}} u^{i} ) d t + ( g^{i}(y^{i}) + v^{i}) d W(t) & & t \in  (0,T), \quad i=1,\cdots,N   ,\\ 
		&y^{0}(t) = y^{N+1}(t) = 0 &&  t \in  (0,T) ,\\ 
		&y^{i}(0) = y^{i}_{0} & &i=1,\cdots,N ,
		\end{aligned}
	\end{cases}
\end{align}
where $ f^{i}(y^{i}) = f(x_{i}, y^{i}) $,  $ g^{i}(y^{i}) = g(x_{i}, y^{i}) $, $ y^{i}_{0} = y_{0}(x_{i}) $, and $ \chi^{i}_{G_{0}} = \chi_{G_{0}}(x_{i}) $.

To present our main results, we need to introduce some notations. 
Define the following regular partition of the interval $[0,1]$ as
$ \mathcal{K} \deq  \{x_i \mid  i=0,1, \ldots, N, N+1 \} $.
For any set of points $\mathcal{W} \subset \mathcal{K}$, we denote the dual meshes $\mathcal{W}^{\prime}$ and $\mathcal{W}^*$, respectively, by
\begin{align*}
	\mathcal{W}^{\prime}\deq\tau_{+}(\mathcal{W}) \cap \tau_{-}(\mathcal{W}), \quad \mathcal{W}^*\deq\tau_{+}(\mathcal{W}) \cup \tau_{-}(\mathcal{W}),
\end{align*}
where
\begin{align*}
	\tau_{ \pm}(\mathcal{W})\deq\bigg \{ x \pm \frac{h}{2}  \bigg \rvert\, x \in \mathcal{W}\bigg\}
    .
\end{align*}
Define $\overline{\mathcal{W}}\deq\left(\mathcal{W}^*\right)^*$ and $\mathring{\mathcal{W}}\deq\left(\mathcal{W}^{\prime}\right)^{\prime}$. 
Notice that for two consecutive points $x_i, x_{i+1} \in \mathcal{W}$, we obtain $x_{i+1} - x_i = h$ provided $\mathring{\overline{\mathcal{W}}} = \mathcal{W}$.
Such a subset $\mathcal{W} \subset \mathcal{K}$ is called a regular mesh.
Denote the boundary of a regular mesh $\mathcal{W}$ as $\partial \mathcal{W}\deq\overline{\mathcal{W}} \backslash \mathcal{W}$.

We define the average operator $A_h$ and the difference operator $D_h$   by
\begin{align*}
	& A_h(u)(t, x)\deq\frac{\tau_{+} u(t, x)+\tau_{-} u(t, x)}{2}, \\
& D_h(u)(t, x)\deq\frac{\tau_{+} u(t, x)-\tau_{-} u(t, x)}{h},
\end{align*}
where $\tau_{ \pm} u(t, x) \deq u\big(t, x \pm \frac{h}{2}\big)$.

Let $L(\mathcal{W})$ represent the set of real-valued functions defined in $\mathcal{W}$.  
We denote by $L_h^2(\mathcal{W})$ the Hilbert space with the inner product 
\begin{align*}
	\langle u, v\rangle_{L_h^2(\mathcal{W})}\deq  \int_{\mathcal{W}} u v\deq h \sum_{x \in \mathcal{W}} u(x) v(x)
\end{align*}
For $ u \in L(\mathcal{W}) $,  its $ L^{\infty}_{h}(\mathcal{W}) $-norm is defined as 
\begin{align*}
	|u|_{L^{\infty}_{h}(\mathcal{W})} \deq \max_{x \in \mathcal{W}} \{ |u(x) | \}.
\end{align*}

With above notations, letting  $ Q \deq (0,T) \times \mathcal{M} $ where $ \mathcal{M} \deq \mathring{\mathcal{K}} $, the controlled semi-discrete stochastic semilinear system \cref{eqDiscreteEq} can be written as 
\begin{align}
	\label{eqDiscreteEqOpe}
	\begin{cases}
		d y - D_{h}^{2} y  d t  = (f(y) + \chi_{G_{0}} u ) d t + ( g(y) + v) d W(t) & \text { in } Q,\\ 
		y(t,0) = y(t,1) = 0 & \text { on } (0,T), \\ 
		y(0, x) = y_{0} & \text { in } \mathcal{M},
	\end{cases}
\end{align}
where the initial datum $ y_{0} \in  L^2_{h}(\mathcal{M}) $ and the controls $(u, v) \in  L^{2}_{\mathbb{F}}(0,T;   L^{2}_{h}(G_{0} \cap \mathcal{M})) \times L^{2}_{\mathbb{F}}(0,T;$  $L^{2}_{h}( \mathcal{M}))$. 
Utilizing the classical theory of stochastic differential equations (see \cite[Theorem 3.2]{Lue2021a}),  there is a unique solution  $ y \in L_{\mathbb{F}}^{2}(\Omega ; C([0, T] ; L^{2}_{h}(\mathcal{M}))) $ to \cref{eqDiscreteEqOpe}.
The main purpose of this paper is to study the null controllability for the semi-discrete stochastic semilinear parabolic equation \cref{eqDiscreteEqOpe}.

There are numerous studies on the controllability theory of deterministic linear semi-discrete parabolic equations (see \cite{Allonsius,Boyer2010,Cerpa2022,Labbe,Nguyen,Zuazua2005,Zuazua2006} and the references cited therein).
However, there are fewer studies on the controllability of semi-discrete  semilinear systems.
In \cite{Boyer2014}, the authors study the null controllability of  semi-discrete semilinear parabolic equations, using the linearization and Schauder topological fixed point theorem. 
In contrast,  there are only a few papers which are concerned with controllability of the semi-discrete stochastic  parabolic equations, i.e., \cite{Zhao} which announced the null controllability of linear  semi-discrete stochastic second order parabolic equations, and \cite{Wang} which announced the null controllability of linear  semi-discrete stochastic fourth order parabolic equations. 
To the best of our knowledge, there are no results available regarding the controllability of semi-discrete stochastic nonlinear  parabolic equations.

For discrete systems, it is well known that  uniform null controllability is hard to address, i.e. for any initial condition, there exist controls $(u,v)$ such the solution of system satisfies $y^i(T ) = 0$ for $ i\in {1, . . . , N }$. 
Thus, we consider the $\phi$-null controllability which proposed in \cite{Boyer2014}, i.e. terminal state can be dominated by a function which tends to zero as space discretization parameter $h$ tends to zero. 
This definition are  highly useful for characterizing the behavior of discrete controls when the discretization parameter goes to zero.

To overcome the lack of compactness for the  solution spaces of stochastic differential equations, we borrow some ideas from \cite{Hernandez2023} for proving the null controllability of  stochastic semilinear parabolic equations \cref{eqDiscreteEqOpe}.  First, we introduce a new weight function that differs from the continuous case in \cite{Hernandez2023}, and we obtain a new global Carleman estimate for a liner backward stochastic semi-discrete parabolic equations with a source term
in a suitable weighted space.
Next, by combining the Carleman estimate with the duality argument and the Hilbert Uniqueness Method introduced in \cite{Boyer2013}, we establish the $\phi$-null controllability for a linear system, which Carleman parameters and discretization 
parameter $h$ are made explicit. 
At last, the $\phi$-null controllability of system \cref{eqDiscreteEqOpe} is ensured through a Banach fixed point method which does not rely on any compactness argument.

The main result of this paper is the following.

\begin{theorem}
    \label{eqNullControllabilityForGeneralEquation}
    Under assumptions $ \mathrm{(A1)-(A5)} $, there exist positive constants $ h_{0} $ and $ C $
    such that for all $ h \leq h_{0} $, there exist $ (u, v) \in  L^{2}_{\mathbb{F}}(0,T;   L^{2}_{h}(G_{0} \cap \mathcal{M})) \times L^{2}_{\mathbb{F}}(0,T; L^{2}_{h}( \mathcal{M}))$, so that the solution $ y $ to \cref{eqDiscreteEqOpe} satisfies 
    \begin{align}
        \label{eqEstimateOfSolutionAndControl}
        \mathbb{E} \int_{Q} |y|^{2} d t 
        + \mathbb{E} \int_{0}^{T} \int_{\mathcal{M} \cap G_{0}} |u|^{2} d t 
        + \mathbb{E} \int_{Q} |v|^{2} d t 
        \leq C  
        \int_{\mathcal{M}} |y_{0}|^{2},
    \end{align}
    and 
    \begin{align}
        \label{eqEstimateOfFinalState}
        \mathbb{E} \int_{\mathcal{M}} |y(T)|^{2} 
        \leq C e^{-\frac{C}{h}}  \int_{\mathcal{M}} |y_{0}|^{2}
        .
    \end{align}
\end{theorem}

Here and in what follows, unless otherwise stated, $ C $ stands for a generic positive constant depending only
on $ G $, $ G_{0} $ and $ T $ whose value may vary from line to line.

\begin{remark}
    \label{rkRemoveG}
    We claim that, under assumptions $ \mathrm{(A1)-(A5)} $, the $ \phi$-null controllability for \cref{eqDiscreteEq} can be reduced to the $ \phi$-null controllability of the following equation:
    \begin{align}
        \label{eqDiscreteWithoutDiffuse}
        \begin{cases}
            d y - D_{h}^{2} y  d t  = (f(y) + \chi_{G_{0}} u ) d t + v d W(t) & \text { in } Q,\\ 
            y(t,0) = y(t,1) = 0 & \text { on } (0,T), \\ 
            y(0, x) = y_{0} & \text { in } \mathcal{M}.
        \end{cases}
    \end{align}
    Indeed, assume that the $ \phi$-null controllability for \cref{eqDiscreteWithoutDiffuse} holds. 
    Then there exist $ (u, v) \in  L^{2}_{\mathbb{F}}(0,T; $  $   L^{2}_{h}(G_{0} \cap \mathcal{M})) \times L^{2}_{\mathbb{F}}(0,T; L^{2}_{h}( \mathcal{M}))$ such that the solution to \cref{eqDiscreteWithoutDiffuse} satisfies \cref{eqEstimateOfSolutionAndControl,eqEstimateOfFinalState}.
    Since the control $ v $ is distributed in the whole domain $ Q $, the state $ y $ satisfies \cref{eqDiscreteEqOpe} with the control $ u^{*} = u $ and $ v^{*} = v - g(y) $.
    Hence \cref{eqEstimateOfFinalState} holds.
    Thanks to assumptions $ \mathrm{(A1), (A4)}, \mathrm{(A5)} $ and \cref{eqEstimateOfSolutionAndControl}, one can show that $ (u^{*}, v^{*}) $ satisfies
    \begin{align*}
        \mathbb{E} \int_{0}^{T} \int_{\mathcal{M} \cap G_{0}} |u^{*}|^{2} d t 
        + \mathbb{E} \int_{Q} |v^{*}|^{2} d t 
        \leq C  
        \int_{\mathcal{M}} |y_{0}|^{2},
    \end{align*}
    which implies the $ \phi$-null controllability for \cref{eqDiscreteEqOpe}.
\end{remark}

\begin{remark}
    From \cref{rkRemoveG}, if we only assume (A1)--(A3), then there exist controls $ (u,v) $ such that \cref{eqEstimateOfFinalState} holds, which is similar to the result for a continuous framework (see \cite[Theorem 1.1]{Hernandez2023}).
    However, without assumptions (A4) and (A5), it is difficult to show that the control of the diffusion term is uniformly bounded with respect to discretization parameter $h$.
    This is why we need these assumptions (A4)--(A5)  in \cref{eqNullControllabilityForGeneralEquation}.
\end{remark}

The rest of this paper is organized as follows. In section 2, we give some preliminaries. A new Carleman estimate is proved in section 3. In section 4, we prove  the controllability for  linear stochastic semi-discrete parabolic equations with a source. Finally, we prove the main controllability result \cref{eqNullControllabilityForGeneralEquation} based on the Banach fixed point method in Section 5.

\section{Some preliminary results}

This section is devoted to presenting  some preliminary results we needed.

\begin{lemma}\cite[Lemma 2.1]{Cerpa2022}
	\label{lemmaLibF}
	For any $u, v \in L(\overline{\mathcal{W}})$, we have for the difference operator
	\begin{align}
		\label{eqDuv}
		D_h(u v)=D_h u A_h v+A_h u D_h v \text {, on } \mathcal{W}^* .
	\end{align}	
	Similarly, the average of the product gives
	\begin{align}
		\label{eqAuv}
		A_h(u v)=A_h u A_h v+\frac{h^2}{4} D_h u D_h v \text {, on } \mathcal{W}^* .
	\end{align}
	Finally, on $\mathring{\mathcal{W}} $ we have
	\begin{align}
		\label{eqA2u}
	u=A_h^2 u-\frac{h^2}{4} D_h^2 u.
	\end{align}
\end{lemma}

\begin{corollary}\cite[Corollary 2.2]{Cerpa2022}
	Let $\overline{\mathcal{W}} \subseteq \mathcal{M}$ be a regular mesh.
    For $u \in L(\overline{\mathcal{W}})$,
	\begin{align}
		\label{eqAu2}
		A_h  (u^2) = (A_h u) ^2+\frac{h^2}{4} (D_h u )^2, \text { on } \mathcal{W}^*.
	\end{align}
	In particular, for all $u \in L(\overline{\mathcal{W}})$,
	\begin{align}
		\label{eqAu2leq}
		A_h(u^2) \geq(A_h u)^2 \text {, on } \mathcal{W}^* .
	\end{align}
	For $u \in L(\overline{\mathcal{W}})$
	\begin{align}
		\label{eqDu2}
		D_h(u^2)=2 D_h u A_h u .
	\end{align}
\end{corollary}

To deal with the boundary conditions, we define the outward normal for $ x \in \partial \mathcal{W} $ as 
\begin{align}
	\label{eqNu}
	\nu(x) \deq \begin{cases}
		1, & \text {if }\tau_{-}(x) \in \mathcal{W}^{*} \text{ and } \tau_{+}(x) \notin \mathcal{W}^{*}, \\
		- 1, & \text {if } \tau_{-}(x) \notin \mathcal{W}^{*} \text{ and } \tau_{+}(x) \in \mathcal{W}^{*}, \\
		0, &  \text{otherwise}.
	\end{cases}
\end{align}
We also introduce the trace operator for $ u \in L(\mathcal{W}^{*}) $ as 
\begin{align*}
	\forall \; x \in \partial \mathcal{W},  \quad  t_{r}(u) \deq 
	\begin{cases}
		\tau_{-} u(x), &  \text {if }  \nu(x)=1, \\
		\tau_{+} u(x), &  \text {if }  \nu(x)=-1, \\
		0, &   \text {if } \nu(x)=0.
	\end{cases}
\end{align*}
Then by the definition of  trace operator, for $ x \in \partial \mathcal{W} $ and $ u \in L(\overline{\mathcal{W}}) $, it is easy to check that
\begin{align}\label{tra}
t_r(|A_hu|^2)\nu-\frac{h^2}{4}t_r(|D_hu|^2)\nu=|u|^2\nu-hut_r(D_hu) .
\end{align}
Finally, let us  introduce the discrete integration on the boundary for $u \in L(\partial \mathcal{W})$ as
\begin{align*}
\int_{\partial \mathcal{W}} u \deq \sum_{x \in \partial \mathcal{W}} u(x) .
\end{align*}

\begin{proposition}\cite[Proposition 2.3]{Cerpa2022}
	Let $\mathcal{W}$ be a semi-discrete regular mesh. For $u \in L(\overline{\mathcal{W}})$ and $v \in L (\mathcal{W}^*)$ we have
	\begin{align}
		\label{eqDIgbp}
		\int_\mathcal{W} u D_h v=-\int_{\mathcal{W}^*} D_h u v+\int_{\partial \mathcal{W}} u t_{r}(v) \nu
	\end{align}
	and
	\begin{align}
		\label{eqAIgbp}
		\int_\mathcal{W} u A_h v=\int_{\mathcal{W}^*} A_h u v-\frac{h}{2} \int_{\partial \mathcal{W}} u t_{r}(v) .
	\end{align}
\end{proposition}

\begin{proposition}\cite[Corollary 2.7]{Cerpa2022}
	\label{propADf}
	Let $ f $ be a $ (n+4) $-times differentiable function defined on $ \mathbb{R} $ and $ m, n \in \mathbb{N} $. Then 
	\begin{align*}
		A_h^m D_h^n f=f^{(n)}+R_{A_h^m} (f^{(n)} )+R_{D_h^n}(f)+R_{A_h^m D_h^n}(f),
	\end{align*}
	where 
	\begin{align*}
		\begin{gathered}
			R_{D_h^n}(f):=h^2 \sum_{k=0}^n\begin{pmatrix}
				n \\
				k
			\end{pmatrix}(-1)^k\bigg(\frac{(n-2 k)}{2}\bigg)^{n+2} \int_0^1 \frac{(1-\sigma)^{n+1}}{(n+1) !} f^{(n+2)}\bigg(\cdot+\frac{(n-2 k) h}{2} \sigma\bigg) d \sigma \\
			R_{A_h^n}(f):=\frac{h^2}{2^{n+2}} \sum_{k=0}^n\begin{pmatrix}
				n \\
				k
			\end{pmatrix}(n-2 k)^2 \int_0^1(1-\sigma) f^{(2)}\bigg(\cdot+\frac{(n-2 k) h}{2} \sigma\bigg) d \sigma
			,
			\end{gathered}
	\end{align*}
	and $ R_{A_h^m D_h^n}:=R_{A_h^m} \circ R_{D_h^n} $. Here, $ f^{(n)} $ denotes the $ n $-th derivative of $ f $ with respect to $ x $.
\end{proposition}

\section{Carleman estimate}
This section is devoted to proving a new Carleman estimate for a backward stochastic semi-discrete parabolic equation. To state our Carleman estimates, we first introduce the weight functions. 

Let $\widetilde{G}$ be an open interval contains $[0,1]$, and $ G_{1}, G_{2}$ be  given two nonempty open
subsets of $G$ satisfying $ G_{2} \subset  \overline{G_{2}} \subset G_{1} \subset  \overline{G_{1}} \subset G_{0} $ and $ \overline{\mathcal{M} \cap G_{2}} \subset \mathcal{M} \cap G_{1}$.
Thanks to \cite[Lemma 1.1]{Fursikov1996}, there exists a function $ \psi \in C^{k}(\overline{\widetilde{G}})$ with $k$ large enough such that
\begin{equation}
	\label{eqPsi}
	0<\psi \le 1 \text{ in } \widetilde{G}, 
	\quad| \partial_{x} \psi| \geq C_{0} >0 \text { in } \overline{ G \backslash G_{2}},
	\quad \partial_{x} \psi(0) > 0,
	  \text{ and }
	\partial_{x}\psi(1) < 0
	.
\end{equation}
Without loss of generality, in what follows,
we assume that $ 0 < T < 1 $.
For some constants $ m \geq 1 $, $ \sigma \geq 2 $ and $ 0 < \delta < \frac{1}{2} $, we define the following weight function 
\begin{align}
    \label{eqWeightedFunctionTime}
    \left\{
        \begin{aligned}
            & \theta(t) = 1 + \bigg( 1- \dfrac{4 t}{T}\bigg)^{\sigma}, && t \in [0, T/4),
            \\
            & \theta(t) = 1, && t \in [T/4, T/2),
            \\
            & \theta \text{ is in increasing on } [T/2, 3T/4),
            \\
            & \theta(t) = \dfrac{1}{(T-t+ \delta T)^{m}},  && t \in [3T/4, T],
            \\
            & \theta \in C^{2}([0,T]).
        \end{aligned}
    \right.
\end{align}
For any parameters $ \lambda \geq 1 $ and $ \mu \geq 1 $, let 
\begin{align}
    \label{eqWeightFuncitonDefine}
    \left\{
        \begin{aligned}
            & \phi(x) = e^{\mu(\psi(x) + 6 m)} - \mu e^{6 \mu(m+1)},
            \quad  \varphi(x) = e^{\mu(\psi(x) + 6 m)}, 
            \\
            & s(t) = \lambda \theta(t),
            \quad  r(t,x) = e^{s(t) \phi(x)}, 
            \quad \rho(t, x)= (r(t,x))^{-1}.
        \end{aligned}
    \right.
\end{align}
We choose $ \sigma $ as 
\begin{align}
    \label{eqSigma}
    \sigma = \lambda \mu^{2} e^{\mu( 6m - 4)}.
\end{align}

\begin{remark}
	In the continuous setting, the regularity of the weight function $\psi$ can be chosen as $C^2([0,1];\mathbb{R})$. Here estimates of discrete derivations of the weight function require knowledge of higher-order derivatives through Taylor formulas. Moreover, to facilitate the Taylor expansion of the weight function at the boundary, the domain of $\psi$ has been slightly extended.
\end{remark}

\begin{remark}
The assumption $ 0 < T < 1 $ is technical. In fact, if $T\ge 1$, we can adjust the values of  $\theta(t)$
 on the interval $(T/2, T]$, and the corresponding Carleman estimate \cref{eqDisCarleman} still holds.
\end{remark}

\begin{remark}
It can be seen from the \cref{eqWeightedFunctionTime} that, compared with the weight function in continuous setting (see \cite{Hernandez2023}), the main difference here is that the weight does not singular at times $t= T$. 
\end{remark}

In the sequel, for any $n\in \mathbb{N}$, we denote by $\mathcal{O}(s^n)$ a function of
order $s^n$ for sufficiently large $s$ (which is independent of $\mu$ and $h$), by $\mathcal{O}_\mu(s^n)$ a function of order $s^n$
for fixed $\mu$ and sufficiently
large $s$ (which is independent of $h$), and
by $\mathcal{O}_\mu((sh)^n)$ a function of
order $(sh)^n$ for fixed $\mu$ and sufficiently
large $s$, sufficiently
small $h$.

We need some estimates regarding the weighted functions.

\begin{lemma}\cite[Lemma 4.3]{Lecaros2021}
	For $ \alpha, \beta \in \mathbb{N} $, we have 
	\begin{align}
		\label{eqprpr} \notag
		\partial_{x}^{\beta} (r \partial_{x}^{\alpha} \rho) 
		& = 
		\alpha^{\beta} (- s \varphi)^{\alpha} \mu^{\alpha+\beta} (\partial_{x} \psi)^{\alpha + \beta}
		+ \alpha \beta(s \varphi)^{\alpha} \mu^{\alpha + \beta - 1} \mathcal{O}(1)
		+ s^{\alpha - 1} \alpha (\alpha - 1) \mathcal{O}_{\mu}(1)
		\\
		& = s^{\alpha} \mathcal{O}_{\mu}(1)
		.
	\end{align}
\end{lemma}

\begin{corollary}\cite[Corollary 4.4]{Lecaros2021}
	For $ \alpha, \beta, \gamma \in \mathbb{N} $, we have 
	\begin{align}
		\label{eqpr2pr} \notag
		\partial_{x}^{\gamma} (r^{2} (\partial_{x}^{\alpha} \rho) \partial_{x}^{\beta} \rho)
		& = 
		(\alpha + \beta)^{\gamma}(- s \varphi)^{\alpha+\beta} \mu^{\alpha+\beta+\gamma} (\partial_{x} \psi)^{\alpha + \beta + \gamma}
		+ \gamma (\alpha + \beta) (s \varphi)^{\alpha + \beta} \mu^{\alpha + \beta + \gamma -1} \mathcal{O}(1)
		\\ \notag
		& \quad 
		+ s^{\alpha + \beta - 1} [\alpha(\alpha-1)+ \beta(\beta-1)] \mathcal{O}_{\mu}(1)
		\\
		& =
		s^{\alpha+\beta}\mathcal{O}_{\mu}(1)
		.
	\end{align}
\end{corollary}
By  \cref{eqWeightFuncitonDefine,eqprpr}, it easy to see that the following result holds.
\begin{lemma}
	For $ \alpha, \beta, \gamma, \sigma\in \mathbb{N} $, we have 
	\begin{align}
		\label{eqpr2pra} \notag
		\partial_{x}^ \sigma(r \partial_x^\alpha \rho \partial_{x}^\beta(r \partial_x^\gamma \rho))
		& = \gamma^\beta
		(\alpha + \gamma)^{\sigma}(- s \varphi)^{\alpha+\gamma} \mu^{\alpha+\beta+\gamma+\sigma} (\partial_{x} \psi)^{\alpha + \beta + \gamma+\sigma}
		+  (s \varphi)^{\alpha + \gamma} \mu^{\alpha + \beta + \gamma+\sigma -1} \mathcal{O}(1)
		\\ \notag
		& \quad 
		+ s^{\alpha + \gamma - 1}  \mathcal{O}_{\mu}(1)
		\\
		& =
		s^{\alpha+\gamma}\mathcal{O}_{\mu}(1)
		.
	\end{align}
\end{lemma}

\begin{lemma}\cite[Lemma 4.8]{Lecaros2021}
	Let $\alpha, j, k,  m, n \in \mathbb{N}$. Provided  $ \lambda h (\delta T)^{-m} \leq 1$, we have
	\begin{align}
		\label{eqdprpr}  
		A_h^j D_h^k \partial_{x}^\alpha(r A_h^m D_h^n \rho)=\partial_x^k \partial_{x}^\alpha(r \partial_x^n \rho)+s^n \mathcal{O}_\mu((s h)^2)=s^n \mathcal{O}_\mu(1)
		.
	\end{align}
\end{lemma}

By \cref{propADf} and \cref{eqdprpr}, it is easy to check that the following estimate holds.
\begin{lemma}
	Let $i, j, k, l, m, n, p, q \in \mathbb{N}$. Provided  $ \lambda h (\delta T)^{-m} \leq 1$, we have
\begin{align}
	\label{eqdprpra} \notag
	A_h^i D_h^j (r A_h^k D_h^l \rho A_h^m D_h^n(r A_{h}^{p} D_{h}^{q}\rho)) & = \partial_{x}^j(r \partial_x^l \rho \partial_{x}^n(r \partial_x^q \rho))+s^{l+q} \mathcal{O}_\mu((s h)^2) \\
	& =s^{l+q} \mathcal{O}_\mu(1).
\end{align}
\end{lemma}

Now, we state the main result of this section, which is a Carleman
estimate for backward stochastic semi-discrete parabolic equations.
\begin{theorem}
    \label{thmCarlemanEstimate}
    Let the functions $ s $ and $ \phi $ be defined in \cref{eqWeightFuncitonDefine}.
    For all $ m \geq 1 $, there exists a constant $ \mu_{0} \geq  1$ such that for any $ \mu \geq \mu_{0} $, one can find positive constants $ C = C(\mu) $, $ h_{0} $ and $ \varepsilon_{0} $ with $ 0 < \varepsilon_{0} \leq 1  $ and $ \lambda_{0} \geq 1 $
    so that for all $ \lambda \geq \lambda_{0} $, $ 0 < h \leq h_{0} $ and 
    $ \lambda h (\delta T)^{-m} \leq \varepsilon_{0} $, it holds that 
    \begin{align}
        \label{eqDisCarleman} \notag
        & \mathbb{E} \int_{Q} s^{3} e^{2 s \phi} |w| ^{2} d t 
        + \mathbb{E} \int_{Q^{*}} s  e^{2 s \phi} |D_{h} w| ^{2} d t 
        + \mathbb{E} \int_{\mathcal{M}} \lambda^{2} \mu^{3}   e^{2 \mu (6m+1) } e^{2 s \phi}   | w| ^{2} \Big|_{t=0}
        \\ \notag
        &
        \leq C \Big(
			\mathbb{E} \int_{0}^{T}\int_{G_{0} \cap \mathcal{M}} s^{3} e^{2 s \phi} |w|^{2} d t
			+ \mathbb{E} \int_{Q} e^{2 s \phi} |f|^{2} d t
			+ \mathbb{E} \int_{Q} s^{2} e^{2 s \phi} |g|^{2} d t
		\\
		& \quad \quad ~ 
		+ h^{-2} \mathbb{E} \int_{\mathcal{M}} e^{2 s \phi}|w|^2 \Big |_{t=T} \Big),
    \end{align}
    for all $ w \in L^{2}_{\mathbb{F}}(\Omega; C([0,T]; L^{2}_{h}(\mathcal{M}))) $  and $ f,  g \in L^{2}_{\mathbb{F}}(0,T;L^{2}_{h}(\mathcal{M})) $ satisfying 
	\begin{align}
		\label{eqW}
		d w + D_{h}^{2} w d t = f d t + g d W(t) 
	\end{align}
	with  $w=0$ on $\partial \mathcal{M}$.
\end{theorem}

\begin{proof}[Proof of \cref{thmCarlemanEstimate}]
	For readability, we have divided the proof into several steps.\\
\emph{Step 1.} 
Letting $ v = r w  $ and recalling $ \rho = r^{-1} $, by \cref{eqDuv}, we have the identity
\begin{align}
    \label{eqrLw}
    r (d w + D_{h}^{2} w d t) = I_{1} d t - \Phi(v) d t - \Psi(v) d t + I_{2},
\end{align}
where
\begin{align}
    \label{eqIij}
    \left\{
        \begin{aligned}
            & I_{1} = r D_{h}^{2} \rho A_{h}^{2} v + r A_{h}^{2} \rho D_{h}^{2} v - \lambda \phi \theta_{t} v + \Phi(v),
            \\
            & I_{2} = d v + 2 r A_{h} D_{h} \rho A_{h} D_{h} v d t + \Psi(v) d  t ,
            \\
            &  \Phi(v) = \frac{h^{2}}{2} D_{h} [D_{h}(r D_{h}^{2} \rho) A_{h} v ],
            \\
            & \Psi(v) = D_{h}( r D_{h} \rho) v.
        \end{aligned}
    \right.
\end{align}
Here,  $\Phi$ and  $\Psi$  are two auxiliary functions to absorb some of the bad terms.
From \cref{eqrLw}, we obtain 
\begin{align*}
    2 rI_{1}(d w + D_{h}^{2} w d t)= 2 I_{1}^{2} d t - 2 I_{1} \Phi(v) d t - 2 I_{1} \Psi(v) d t + 2 I_{1} I_{2}
    ,
\end{align*}
which, combined with Cauchy-Schwarz inequality, and noting \cref{eqW}, implies
\begin{align}
    \label{eqPCarS1e1}
    \mathbb{E} \int_{Q} |rf|^{2} d t 
    \geq 
    2 \mathbb{E} \int_{Q} I_{1} I_{2} 
    - 2 \mathbb{E} \int_{Q} |\Phi(v)|^{2} d t 
    - 2 \mathbb{E} \int_{Q} |\Psi(v)|^{2} d t 
    .
\end{align}
In the following steps, we estimate the right-hand side terms of \cref{eqPCarS1e1}. 
Hereinafter, 
for convenience, we put 
\begin{align}
	\label{eqDefIij}
	\mathbb{E} \int_{Q} I_{1} I_{2} = \sum_{i=1}^{4} \sum_{j=1}^{3} I_{i j},
\end{align}
where $ I_{i j} $ stands for the inner product in $ L^{2}_{\mathbb{F}}(0,T;L_{h}^{2}(\mathcal{M})) $ between the $ i $-th term of $ I_{1} $ and $ j $-th term of $ I_{2} $.

In order to shorten the formulas used in the sequel, we define the following
\begin{align*}
	\mathcal{A}_{1} &  =
	\mathbb{E} \int_{Q} [s^{2} \mathcal{O}_{\mu}(1) + s^{3} \mathcal{O}_{\mu}((sh)^{2}) +  s^{3} \mu^{3} \varphi^{3} \mathcal{O}(1)] |v|^{2} d t
	\\
	& \quad 
	+ \mathbb{E} \int_{Q^{*}} [  \mathcal{O}_{\mu}(1) + s  \mathcal{O}_{\mu}((sh)^{2}) + s \mu \varphi \mathcal{O}(1)] |D_{h} v|^{2} d t
	,
	\\
	\mathcal{A}_{2} & = 
	\mathbb{E} \int_{Q} s^{2} \mathcal{O}_{\mu}(1) |d v|^{2}
	+ \mathbb{E} \int_{ Q^{*}}  \mathcal{O}_{\mu}((s h)^{2}) |D_{h}  (d v)|^{2}
	,
	\\
	\mathcal{A}_{3} & = 
	\mathbb{E} \int_{\mathcal{M}} s^{2} \mathcal{O}_{\mu}(1) |v|^{2} \Big|_{t=T}
	+ \mathbb{E} \int_{\mathcal{M}^{*}}  \mathcal{O} (1) |D_{h} v|^{2} \Big|_{t=T}
    \\
    & \quad 
    + \mathbb{E} \int_{\mathcal{M}} (  s^{2} \mu^{2}  e^{2\mu(6m+1)} \mathcal{O}(1) +  s \mathcal{O}_{\mu}(1) + s^{2} \mathcal{O}_{\mu}((sh)^{2})) |v|^{2} \Big|_{t=0}
	,
	\\
	\mathcal{B} & = 
	\mathbb{E} \int_{\partial Q} [  \mathcal{O}_{\mu}(1) + s  \mathcal{O}_{\mu}((sh)^{2}) ] t_{r} (|D_{h} v|^{2}) d t
	.
\end{align*}

\emph{Step 2.} We compute $ \sum\limits_{i=1}^{4} I_{i1} $.
Combining \cref{eqIij,eqA2u}, letting 
$q^{11} = r D_{h}^{2} \rho$,
we obtain 
\begin{align}
    \label{eqIi1}
    \notag
    2\sum_{i=1}^{4} I_{i 1}  
    & =
    2 \mathbb{E} \int_{Q} (r D_{h}^{2} \rho A_{h}^{2} v + r A_{h}^{2} \rho D_{h}^{2} v - \lambda \phi \theta_{t} v + \Phi(v)) d v 
    \\ \notag
    & = 
    2 \mathbb{E} \int_{Q} \bigg(r D_{h}^{2} \rho v + D_{h}^{2} v + \frac{h^{2}}{2} r D_{h}^{2} \rho D_{h}^{2} v  - \lambda \phi \theta_{t} v + \Phi(v)\bigg) d v 
    \\
    & = 
    2 \mathbb{E} \int_{Q} \bigg(q^{11} v + D_{h}^{2} v  + \frac{h^{2}}{2} q^{11} D_{h}^{2} v   - \lambda \phi \theta_{t} v+ \Phi(v)\bigg) d v 
    .
\end{align}

Thanks to It\^o's formula, we have 
\begin{align}
    \notag 
    \label{eqI11e1}
    2 \mathbb{E} \int_{Q} (q^{11}  - \lambda \phi \theta_{t} ) v d v 
    &=
    \mathbb{E} \int_{\mathcal{M}} (q^{11}   - \lambda \phi \theta_{t}  ) v^{2} \Big|_{t=T}
    - \mathbb{E} \int_{\mathcal{M}} (q^{11}   - \lambda \phi \theta_{t}  ) v^{2} \Big|_{t=0}
    \\
    & \quad 
    - \mathbb{E} \int_{Q} \partial_{t} (q^{11}   - \lambda \phi \theta_{t}  ) v^{2}  d t
    - \mathbb{E} \int_{Q}   (q^{11}   - \lambda \phi \theta_{t}  ) | d v|^{2}
    .
\end{align}
Combining \cref{eqdprpr,eqprpr}, we conclude that
\begin{align}
	\label{eqQ11}
	q^{11} = s^{2} \mu^{2} \varphi^{2} |\partial_{x} \psi|^{2} + s \mathcal{O}_{\mu}(1) + s^{2} \mathcal{O}_{\mu}((sh)^{2})
    .
\end{align}
Using definition  of $\theta$ in \cref{eqWeightedFunctionTime}, we obtain $ \theta_{t}(T) = m (\delta T)^{-m-1} $. Thus, we get from \cref{eqQ11,eqWeightFuncitonDefine} that 
\begin{align}
    \label{eqI11e2}
    \mathbb{E} \int_{\mathcal{M}} (q^{11}   - \lambda \phi \theta_{t}  ) v^{2} \Big|_{t=T} 
    \geq 
    - \mathbb{E} \int_{\mathcal{M}} s^{2} \mathcal{O}_{\mu}(1) |v|^{2} \Big|_{t=T} 
    \geq 
    - \mathcal{A}_{3}.
\end{align}
 Noting $\theta_{t}(0)=-\frac{4 \sigma}{T}$, and  from \cref{eqWeightFuncitonDefine,eqSigma}, a direct computation yields
\begin{align*}
    \lambda \phi \theta_{t}(0) = \frac{4}{T} \lambda^2 \mu^{2} e^{\mu (6 m - 4)} \left(\mu e^{6 \mu(m+1)}-e^{\mu(\psi(x) + 6 m)}\right)\geq c \lambda^{2} \mu^{3} e^{2 \mu (6 m + 1)},
\end{align*}
for all $\mu>1$ and some constant $c>0$ uniform with respect to $T$.
Hence,  using \cref{eqQ11} it can be readily seen that
\begin{align}
    \label{eqI11e3}
    & - \mathbb{E} \int_{\mathcal{M}} (q^{11}   - \lambda \phi \theta_{t}  ) v^{2} \Big|_{t=0} 
 \geq
    c \mathbb{E} \int_{\mathcal{M}}  \lambda^{2} \mu^{3} e^{2 \mu (6 m + 1)} |v(0)|^{2} 
    - \mathcal{A}_{3}
    .
\end{align}

Thanks to \cref{eqWeightFuncitonDefine,eqQ11}, we see that 
\begin{align*}
    \partial_{t} q^{11} = \frac{\theta_{t}}{\theta} (2s^{2} \mu^{2} \varphi^{2} |\partial_{x} \psi|^{2} + s \mathcal{O}_{\mu}(1) + s^{2} \mathcal{O}_{\mu}((sh)^{2})), \quad \partial_{t}(\lambda \phi \theta_{t}) = s \phi \frac{\theta_{tt}}{\theta},
\end{align*}
which implies that 
\begin{align}
	\notag
    \label{eqI11e4a}
    &- \mathbb{E} \int_{Q} \partial_{t} (q^{11}   - \lambda \phi \theta_{t}  ) v^{2}  d t\\
    &=
    - \mathbb{E} \int_{Q}   \frac{\theta_{t}}{\theta}  (2s^{2} \mu^{2} \varphi^{2} |\partial_{x} \psi|^{2} + s \mathcal{O}_{\mu}(1) + s^{2} \mathcal{O}_{\mu}((sh)^{2}))v^{2}  d t
    +\mathbb{E} \int_{Q}   s\phi \frac{\theta_{tt}}{\theta}  v^{2}  d t
    .
\end{align}
Using the definition of $\theta$ in \cref{eqWeightedFunctionTime}, it is not difficult to see that for $t\in [0, \frac{T}{4}]$, $\theta\in [1,2]$, and 
$ |\theta_{tt}|\le C\lambda^2 \mu^4e^{2\mu(6m-4)}.
$
Noting there exists $ \mu_{1} \geq 1 $ such that for all $ \mu \geq \mu_{1} $, it holds that $\mu^2e^{-2\mu}\le 1$.
Thus, by the definitions of $\phi,\varphi$ in \cref{eqWeightFuncitonDefine}, for $t\in [0, \frac{T}{4}]$,  we end up with
$$ s\frac{|\phi\theta_{tt}|}{\theta}\le Cs^3 \mu^5e^{2\mu(6m-4)}e^{6\mu(m+1)}=C s^3 \mu^5e^{\mu(18m-2)}\le Cs^3\mu^3\varphi^3.
$$
Moreover, for $t\in [\frac{T}{4}, T]$, we have $|\theta_{tt}|\le C\theta^3$,  thus it is easy to check that
$$
s\frac{|\phi\theta_{tt}|}{\theta}\le s\mu\theta^2e^{6\mu(m+1)}\le s^3\mu\varphi^3.
$$
Therefore, by  \cref{eqI11e4a}, we conclude that
\begin{align}
	\notag
	\label{eqI11e4}
	&- \mathbb{E} \int_{Q} \partial_{t} (q^{11}   - \lambda \phi \theta_{t}  ) v^{2}  d t\\
	&\ge
	- \mathbb{E} \int_{Q}   \frac{\theta_{t}}{\theta}  (2s^{2} \mu^{2} \varphi^{2} |\partial_{x} \psi|^{2} + s \mathcal{O}_{\mu}(1) + s^{2} \mathcal{O}_{\mu}((sh)^{2}))v^{2}  d t
	- \mathcal{A}_{1}
	.
\end{align}

From \cref{eqWeightedFunctionTime}, it easy to check that $|\theta_{t}|\le \sigma=\lambda\mathcal{O}_{\mu}(1)$ for $t\in [0, \frac{T}{4})$, and $|\theta_{t}|\le C\theta^2$ for $t\in [\frac{T}{4}, T]$. Then 
$ |\lambda \phi \theta_{t} | \leq s ^{2} \mathcal{O}_{\mu}(1)$ for $t\in [0, T]$.  And noting
 \cref{eqQ11}, we deduce that
\begin{align}
    \label{eqI11e5}
    - \mathbb{E} \int_{Q}   (q^{11}   - \lambda \phi \theta_{t}  ) | d v|^{2} 
    \geq 
    - \mathbb{E} \int_{Q}  s ^{2} \mathcal{O}_{\mu}(1)   | d v|^{2}  
    \geq
    - \mathcal{A}_{2}
    .
\end{align}
Combining \cref{eqI11e1,eqI11e2,eqI11e3,eqI11e4,eqI11e5}, for $ \mu \geq \mu_{1}  $ and $ \lambda h (\delta T)^{-m} \leq 1$, we conclude that 
\begin{align}
    \label{eqIi2} \notag
    2 \mathbb{E} \int_{Q} (q^{11}  - \lambda \phi \theta_{t} ) v d v
    &\geq 
    - \mathbb{E} \int_{Q}   \frac{\theta_{t}}{\theta}  (2s^{2} \mu^{2} \varphi^{2} |\partial_{x} \psi|^{2} + s \mathcal{O}_{\mu}(1) + s^{2} \mathcal{O}_{\mu}((sh)^{2}))v^{2}  d t\\
    & \quad +c\mathbb{E} \int_{\mathcal{M}}  \lambda^{2} \mu^{3} e^{2 \mu (6 m + 1)} |v(0)|^{2}- \mathcal{A}_{1}
    - \mathcal{A}_{2} - \mathcal{A}_{3}
    .
\end{align}

From \cref{eqDIgbp} and It\^o's formula, and noting $ v = 0 $ on $ \partial Q $, we obtain 
\begin{align}
    \label{eqIi3} \notag
    2 \mathbb{E} \int_{Q} D_{h}^{2} v d v 
    & = 
    - 2 \mathbb{E} \int_{Q^{*}} D_{h}  v D_{h}(d v ) 
    + 2  \mathbb{E} \int_{\partial Q} d v t_{r} (D_{h} v)  \nu
    \\ \notag
    & = 
    -  \mathbb{E} \int_{\mathcal{M}^{*}} | D_{h}  v(T)|^{2} 
    +  \mathbb{E} \int_{\mathcal{M}^{*}} | D_{h}  v(0)|^{2} 
    +   \mathbb{E} \int_{Q^{*}} | D_{h}(d v ) |^{2}
    \\
    & \geq   \mathbb{E} \int_{\mathcal{M}^{*}} | D_{h}  v(0)|^{2}+
    \mathbb{E} \int_{Q^{*}} | D_{h}(d v ) |^{2} - \mathcal{A}_{3}
    .
\end{align}

Thanks to \cref{eqDIgbp,eqDuv} and $ v = 0 $ on $ \partial Q $, we have 
\begin{align}
    \label{eqI31e1}
    h^{2} \mathbb{E} \int_{Q} q^{11} D_{h}^{2} v d v 
    = 
    - h^{2} \mathbb{E} \int_{Q^{*}} A_{h} q^{11} D_{h} (d v) D_{h} v  
    - h^{2} \mathbb{E} \int_{Q^{*}} D_{h} q^{11} A_{h} (d v) D_{h} v  
    .
\end{align}
From It\^o's formula, we obtain 
\begin{align}
    \notag
    \label{eqI31e2}
    - h^{2} \mathbb{E} \int_{Q^{*}} A_{h} q^{11} D_{h} (d v) D_{h} v  
    & =
    - \frac{h^{2}}{2} \mathbb{E} \int_{Q^{*}} d ( A_{h} q^{11} |D_{h} v|^{2}) 
    + \frac{h^{2}}{2} \mathbb{E} \int_{Q^{*}} \partial_{t} (A_{h} q^{11}) |D_{h} v|^{2}
    \\
    & \quad 
    + \frac{h^{2}}{2} \mathbb{E} \int_{Q^{*}}    A_{h} q^{11}  |D_{h} (d v)|^{2}
    .
\end{align}
By It\^o's formula and \cref{eqDuv}, a direct computation shows that
\begin{align*}
	 A_{h} (d v) D_{h} v=\frac{1}{2}d(D_{h}(v^{2})) 
	- A_{h} v D_{h}(d v)
	-  \frac{1}{2} D_{h} ((d v)^{2})
	,
\end{align*}
which  combined with \cref{eqIij}, \cref{eqDIgbp} and $ t_r(v) = 0 $ on $ \partial Q^* $ implies 
\begin{align}
    \notag
    \label{eqI31e3}
    &
    - h^{2} \mathbb{E} \int_{Q^{*}} D_{h} q^{11} A_{h} (d v) D_{h} v   
    \\ \notag
    & =
    - \frac{h^{2}}{2} \mathbb{E} \int_{Q^{*}} D_{h} q^{11} d (D_{h}(v^{2}))  
    + h^{2}  \mathbb{E} \int_{Q^{*}} D_{h} q^{11}  A_{h} v D_{h} (d v)
    +  \frac{h^{2}}{2}  \mathbb{E} \int_{Q^{*}} D_{h} q^{11}  D_{h}((d v)^{2})   
    \\ \notag
    & =
     \frac{h^{2}}{2} \mathbb{E} \int_{Q} D^{2}_{h} q^{11} d ( v^{2}) - h^{2}  \mathbb{E} \int_{Q} D_h(D_{h} q^{11}  A_{h} v) d v
    +  \frac{h^{2}}{2}  \mathbb{E} \int_{Q^{*}} D_{h} q^{11}  D_{h}((d v)^{2})
    \\
    & = 
     \frac{h^{2}}{2} \mathbb{E} \int_{Q } d ( D^{2}_{h} q^{11}  v^{2} ) 
    - \frac{h^{2}}{2} \mathbb{E} \int_{Q } \partial_{t} ( D_{h}^{2} q^{11})   v^{2}  d t - h^{2}  \mathbb{E} \int_{Q} D_h(D_{h} q^{11}  A_{h} v) d v
    -  \frac{h^{2}}{2}  \mathbb{E} \int_{Q} D_{h}^{2} q^{11}    |d v|^{2}
    .
\end{align}
Combining \cref{eqI31e1,eqI31e2,eqI31e3}, and noting \cref{eqIij}, $q^{11} = r D_{h}^{2} \rho$, it holds that 
\begin{align}
    \notag
    \label{eqI31e4}
    & 2 \mathbb{E} \int_{Q} \bigg(  \frac{h^{2}}{2} q^{11} D_{h}^{2} v   + \Phi(v)\bigg) d v 
    \\ \notag
    & =
    -\frac{h^{2}}{2} \Big[
        \mathbb{E} \int_{Q^{*}} d ( A_{h} q^{11} |D_{h} v|^{2}) 
        -   \mathbb{E} \int_{Q^{*}} \partial_{t} (A_{h} q^{11}) |D_{h} v|^{2}    
        - \mathbb{E} \int_{Q^{*}}    A_{h} q^{11}  |D_{h} (d v)|^{2}
    \\
    & \quad
    - \mathbb{E} \int_{Q } d ( D^{2}_{h} q^{11}   v^{2} ) 
    + \mathbb{E} \int_{Q}  \partial_{t} ( D_{h}^{2} q^{11})   v^{2}  d t 
    +  \mathbb{E} \int_{Q} D_{h}^{2} q^{11}    |d v|^{2}
    \Big]
    .
\end{align}

By $q^{11} = r D_{h}^{2} \rho$,  from \cref{eqdprpr,eqprpr} and $ |D_{h} v|^{2} \leq h^{-2} (|\tau_{+}v|^{2} + |\tau_{-} v|^{2}) $, we get 
\begin{align}
    \label{eqI31e5} \notag
    & 
    -\frac{h^{2}}{2}  \mathbb{E} \int_{Q^{*}} d ( A_{h} q^{11} |D_{h} v|^{2})
    \\ \notag
    & =
    -\frac{h^{2}}{2}  \mathbb{E} \int_{\mathcal{M}^{*}}   A_{h} q^{11} |D_{h} v|^{2} \Big|_{t=T}
    +\frac{h^{2}}{2}  \mathbb{E} \int_{\mathcal{M}^{*}}   A_{h} q^{11} |D_{h} v|^{2} \Big|_{t=0}
    \\ \notag
    & \geq 
    - \mathbb{E} \int_{\mathcal{M}^{*}} \mathcal{O}_{\mu}((s h)^{2}) |D_{h} v|^{2} \Big|_{t=T}
    - \mathbb{E} \int_{\mathcal{M}} (  s^{2} \mu^{2}  e^{2\mu(6m+1)} \mathcal{O}(1) +  s \mathcal{O}_{\mu}(1) + s^{2} \mathcal{O}_{\mu}((sh)^{2})) |v|^{2} \Big|_{t=0}
    \\
    & \geq 
    - \mathcal{A}_{3}
    ,
\end{align}
and 
\begin{align}
    \label{eqI31e6}
    \frac{h^{2}}{2}  \mathbb{E} \int_{Q^{*}} d ( D^{2}_{h} q^{11}   v^{2} ) 
    & \geq 
    - \mathbb{E} \int_{\mathcal{M} } \mathcal{O}_{\mu}((s h)^{2}) |v|^{2} \Big|_{t=T}
    - \mathbb{E} \int_{\mathcal{M}}   \mathcal{O}_{\mu}((sh)^{2})  |v|^{2} \Big|_{t=0}
    \geq 
    - \mathcal{A}_{3}
    .
\end{align}
Thanks to \cref{eqWeightFuncitonDefine,eqdprpr,eqprpr}, noting $q^{11} = r D_{h}^{2} \rho$,  we see that 
\begin{align*}
    \partial_{t} (A_{h} q^{11})= \frac{\theta_{t}}{\theta} s^{2} \mathcal{O}_{\mu}(1), \ \partial_{t} (D_{h}^{2} q^{11})= \frac{\theta_{t}}{\theta} s^{2} \mathcal{O}_{\mu}(1),
\end{align*}
and noting 
$| \theta_{t} | \le \lambda \theta^2 \mathcal{O}_{\mu}(1)$,
which implies that, for $ \lambda h (\delta T)^{-m} \leq 1 $,
\begin{align}
    \label{eqI31e7} \notag
    & 
    -\frac{h^{2}}{2} \Big[ 
        -   \mathbb{E} \int_{Q^{*}} \partial_{t} (A_{h} q^{11}) |D_{h} v|^{2}
        + \mathbb{E} \int_{Q}  \partial_{t} ( D_{h}^{2} q^{11})   v^{2}  d t
    \Big]
    \\ \notag
    & \geq 
    - \mathbb{E} \int_{ Q^{*} }  s  \mathcal{O}_{\mu}((sh)^{2})  |D_{h} v|^{2}  d t
    - \mathbb{E} \int_{Q}   s  \mathcal{O}_{\mu}((sh)^{2})   |v|^{2} d t
    \\
    & \geq 
    -\mathcal{A}_{1}
    .
\end{align}
It follows from \cref{eqprpr}, \cref{eqdprpr} and $q^{11} = r D_{h}^{2} \rho$ that, for $ \lambda h (\delta T)^{-m} \leq 1 $,
\begin{align}
    \label{eqI31e8} \notag
    & 
    -\frac{h^{2}}{2} \Big[
        - \mathbb{E} \int_{Q^{*}}    A_{h} q^{11}  |D_{h} (d v)|^{2}
        + \mathbb{E} \int_{Q} D_{h}^{2} q^{11}    |d v|^{2}
    \Big]
    \\ \notag
    & \geq  
    -\mathbb{E} \int_{Q}  \mathcal{O}_{\mu}((s h)^{2}) |d v|^{2}
	- \mathbb{E} \int_{ Q^{*}}  \mathcal{O}_{\mu}((s h)^{2}) |D_{h}  (d v)|^{2}
    \\
    & \geq 
    - \mathcal{A}_{2}
    .
\end{align}
Combining \cref{eqI31e4,eqI31e5,eqI31e6,eqI31e7,eqI31e8}, we conclude that, for $ \lambda h (\delta T)^{-m} \leq 1 $,
\begin{align}
    \label{eqIi4}
    2 \mathbb{E} \int_{Q} \bigg(  \frac{h^{2}}{2} q^{11} D_{h}^{2} v   + \Phi(v)\bigg) d v 
    \geq - \mathcal{A}_{1}
    - \mathcal{A}_{2}
    - \mathcal{A}_{3}
    .
\end{align}

Thus, from \cref{eqIi1,eqIi2,eqIi3,eqIi4}, for $ \mu \geq \mu_{1} $ and $ \lambda h (\delta T)^{-m} \leq 1 $, we see that
\begin{align}
    \label{eqIi1Final} \notag
    2\sum_{i=1}^{4} I_{i 1} 
    & \geq 
    c \mathbb{E} \int_{\mathcal{M}}  \lambda^{2} \mu^{3} e^{2 \mu (6 m + 1)} |v(0)|^{2}+\mathbb{E}\int_{\mathcal{M}^{*}} | D_{h}  v(0)|^{2}
    + \mathbb{E} \int_{Q^{*}} | D_{h}(d v ) |^{2} - \mathcal{A}_{1}
    - \mathcal{A}_{2} - \mathcal{A}_{3}
    \\
    & \quad 
    - \mathbb{E} \int_{Q}   \frac{\theta_{t}}{\theta}  (2s^{2} \mu^{2} \varphi^{2} |\partial_{x} \psi|^{2} + s \mathcal{O}_{\mu}(1) + s^{2} \mathcal{O}_{\mu}((sh)^{2}))v^{2}  d t
    .
\end{align}

\emph{Step 3.} We compute $ \sum\limits_{i=1}^{4} I_{i2} $.

From \cref{eqIij,eqDu2,eqDIgbp,tra,eqAu2,eqAIgbp,eqdprpra,eqpr2pr}, letting $ q^{12} = 2 r^{2} D_{h}^{2} \rho A_{h} D_{h} \rho $, for $ \lambda h (\delta T)^{-m} \leq 1 $, we arrive at the following estimate 
\begin{align}
    \label{eqI12} \notag
    2I_{12}
    & = 
    \mathbb{E} \int_{Q} q^{12} D_{h}((A_{h} v)^{2}) d t 
    \\ \notag
    & = 
    - \mathbb{E} \int_{Q^{*}} D_{h} q^{12}  (A_{h} v)^{2}  d t 
    + \mathbb{E} \int_{ \partial Q }  q^{12}  t_{r}((A_{h} v)^{2}) \nu d t 
    \\ \notag
    & = 
    - \mathbb{E} \int_{Q^{*}} D_{h} q^{12}   A_{h} (v ^{2})  d t 
    + \frac{h^{2}}{4} \mathbb{E} \int_{Q^{*}} D_{h} q^{12}   |D_{h} v| ^{2}   d t 
    + \frac{h^{2}}{4} \mathbb{E} \int_{ \partial Q }  q^{12}  t_{r}(|D_{h} v|^{2}) \nu d t 
    \\ \notag
    & = 
    - \mathbb{E} \int_{Q} A_{h} D_{h} q^{12}    |v| ^{2}  d t 
    + \frac{h^{2}}{4} \mathbb{E} \int_{Q^{*}} D_{h} q^{12}   |D_{h} v| ^{2}   d t 
    + \frac{h^{2}}{4} \mathbb{E} \int_{ \partial Q }  q^{12}  t_{r}(|D_{h} v|^{2}) \nu d t 
    \\ \notag
    & \geq  
    6 \mathbb{E} \int_{Q} s^{3} \mu^{4} \varphi^{3} |\partial_{x} \psi|^{4}  |v| ^{2}  d t 
    -  \mathbb{E} \int_{Q} [s^{2} \mathcal{O}_{\mu}(1) + s^{3} \mathcal{O}_{\mu}((sh)^{2}) +  s^{3} \mu^{3} \varphi^{3} \mathcal{O}(1)]  |v| ^{2}  d t 
    \\ \notag
    & \quad 
    - \mathbb{E} \int_{Q^{*}}   s  \mathcal{O}_{\mu}((sh)^{2})  |D_{h} v|^{2} d t
    - \mathbb{E} \int_{\partial Q}   s  \mathcal{O}_{\mu}((sh)^{2})   t_{r} (|D_{h} v|^{2}) d t
    \\
    & \geq  
    6 \mathbb{E} \int_{Q} s^{3} \mu^{4} \varphi^{3} |\partial_{x} \psi|^{4}  |v| ^{2}  d t 
    -  \mathcal{A}_{1} - \mathcal{B}
    .
\end{align}

Thanks to \cref{eqDu2,eqDIgbp,eqdprpra,eqprpr}, letting $ q^{22} = 2 r^{2} A_{h}^{2} \rho A_{h} D_{h} \rho $, for $ \lambda h (\delta T)^{-m} \leq 1 $,  a   direct computation gives 
\begin{align}
    \label{eqI22} \notag
    2I_{22}
    & = 
    \mathbb{E} \int_{Q} q^{22} D_{h} ( (D_{h} v)^{2} ) d t 
    \\ \notag
    & = 
    - \mathbb{E} \int_{Q^{*}}  D_{h} q^{22}  |D_{h} v|^{2}  d t 
    + \mathbb{E} \int_{\partial Q}   q^{22} t_{r} ( |D_{h} v|^{2} ) \nu d t 
    \\ \notag
    & \geq  
    2 \mathbb{E} \int_{Q^{*}}  s \mu^{2} \varphi |\partial_{x} \psi|^{2}   |D_{h} v|^{2}  d t 
    - 2 \mathbb{E} \int_{\partial Q} s \mu \varphi  \partial_{x} \psi t_{r} ( |D_{h} v|^{2} ) \nu d t 
    \\ \notag
    & \quad 
    - \mathbb{E} \int_{Q^{*}} [  s  \mathcal{O}_{\mu}((sh)^{2}) + s \mu \varphi \mathcal{O}(1)] |D_{h} v|^{2} d t 
    - \mathbb{E} \int_{\partial Q}  s  \mathcal{O}_{\mu}((sh)^{2})   t_{r} (|D_{h} v|^{2}) d t
    \\
    & \geq  
    2 \mathbb{E} \int_{Q^{*}}  s \mu^{2} \varphi |\partial_{x} \psi|^{2}   |D_{h} v|^{2}  d t 
    - 2 \mathbb{E} \int_{\partial Q} s \mu \varphi  \partial_{x} \psi t_{r} ( |D_{h} v|^{2} ) \nu d t
    - \mathcal{A}_{1} - \mathcal{B}
    .
\end{align}

Combining \cref{eqDIgbp,eqDuv,eqDu2,eqAu2,eqAIgbp}, letting $ q^{32} = 2\lambda \phi \theta_{t} r A_{h} D_{h} \rho $, we see that 
\begin{align*}
    2I_{32} 
    & = 
    -2  \mathbb{E} \int_{Q} q^{32} A_{h} D_{h} v v d t 
    \\
    & = 
    2 \mathbb{E} \int_{Q^{*}} D_{h} q^{32} |A_{h}  v|^{2}   d t 
    + 2  \mathbb{E} \int_{Q^{*}} A_{h} q^{32}  A_{h}  v D_{h} v    d t 
    \\
    & = 
    2  \mathbb{E} \int_{Q^{*}} D_{h} q^{32} |A_{h}  v|^{2}   d t 
    -  \mathbb{E} \int_{Q} D_{h} A_{h} q^{32}   v^{2}    d t 
    \\
    & = 
      \mathbb{E} \int_{Q} A_{h} D_{h}  q^{32} | v|^{2}   d t 
    - \frac{h^{2}}{2} \mathbb{E} \int_{Q^{*}} D_{h}  q^{32}   |D_{h} v|^{2}    d t 
    .
\end{align*}
By \cref{propADf}, \cref{eqprpr}, \cref{eqdprpr}, and noting 
$| \theta_{t} | \le \lambda \theta^2 \mathcal{O}_{\mu}(1)$, we obtain 
\begin{align*}
    A_{h} D_{h} q^{32} 
    & = 
    \partial_{x}(2\lambda \phi \theta_{t} r A_{h} D_{h} \rho) 
    + \frac{\theta_{t}}{\theta} s\mathcal{O}_{\mu}((s h)^{2})
    \\
    & = 
      2\lambda \phi \theta_{t} \partial_{x}( r A_{h} D_{h} \rho )
    +2\partial_{x}(\lambda \phi \theta_{t}) r A_{h} D_{h} \rho 
    + \frac{\theta_{t}}{\theta} s\mathcal{O}_{\mu}((s h)^{2})
    \\
    & = 
    2\lambda \phi \theta_{t} \partial_{x}( r \partial_{x} \rho )
    + \frac{\theta_{t}}{\theta} s^{2} \mathcal{O}_{\mu}((s h)^{2})
    - 2\frac{\theta_{t}}{\theta} s^{2} \mu^{2} \varphi^{2} (\partial_{x} \psi)^{2}
    + \frac{\theta_{t}}{\theta} s\mathcal{O}_{\mu}((s h)^{2})
    \\
    & = 
    2\lambda \phi \theta_{t} \partial_{x}( r \partial_{x} \rho )
    - 2\frac{\theta_{t}}{\theta} s^{2} \mu^{2} \varphi^{2} (\partial_{x} \psi)^{2}
    + \frac{\theta_{t}}{\theta} s^{2} \mathcal{O}_{\mu}((s h)^{2})
    ,
\end{align*}
 and $D_{h} q^{32} =s^3\mathcal{O}_{\mu}(1)$,
which implies that, for $ \lambda h (\delta T)^{-m} \leq 1$,
\begin{align}
    \label{eqI32} 
    2I_{32}
    & \geq 
    2 \mathbb{E} \int_{Q} \bigg[
        \lambda \phi \theta_{t} \partial_{x}( r \partial_{x} \rho )
    - \frac{\theta_{t}}{\theta} s^{2} \mu^{2} \varphi^{2} (\partial_{x} \psi)^{2}
    + \frac{\theta_{t}}{\theta} s^{2} \mathcal{O}_{\mu}((s h)^{2})
    \bigg] v^{2} d t 
    - \mathcal{A}_{1}
    .
\end{align}

From \cref{eqIij,eqDuv,eqDu2,eqDIgbp,tra,eqAu2leq,eqdprpra}, for $ \lambda h (\delta T)^{-m} \leq 1$, we obtain 
\begin{align}
    \label{eqI42}\notag
    2I_{42}
    & = 
    2 h^{2} \mathbb{E} \int_{Q}  D_{h} [D_{h}(r D_{h}^{2} \rho) A_{h} v ]  r A_{h} D_{h} \rho A_{h} D_{h} v d t
    \\ \notag
    & = 
    2 h^{2} \mathbb{E} \int_{Q}  D_{h}^{2}(r D_{h}^{2} \rho)    r A_{h} D_{h} \rho A_{h}^{2} v A_{h} D_{h} v d t
    + 2 h^{2} \mathbb{E} \int_{Q} A_{h} D_{h} (r D_{h}^{2} \rho)    r A_{h} D_{h} \rho   |A_{h} D_{h} v|^{2} d t
    \\ \notag
    & =
    - h^{2} \mathbb{E} \int_{Q^{*}}  D_{h}[D_{h}^{2}(r D_{h}^{2} \rho)    r A_{h} D_{h} \rho] | A_{h} v |^{2}  d t
    + \frac{h^{4}}{4} \mathbb{E} \int_{ \partial Q }   D_{h}^{2}(r D_{h}^{2} \rho)    r A_{h} D_{h} \rho t_{r}( | D_{h} v |^{2} ) \nu d t
    \\ 
    & \quad 
    + 2 h^{2} \mathbb{E} \int_{Q} A_{h} D_{h} (r D_{h}^{2} \rho)    r A_{h} D_{h} \rho   |A_{h} D_{h} v|^{2} d t
    .
\end{align}
From \cref{eqAIgbp,eqAu2leq,eqdprpra}, it is easy to check that the first term of the right hand of  \cref{eqI42} can be bound  by
\begin{align*}
	 - h^{2} \mathbb{E} \int_{Q^{*}}  D_{h}[D_{h}^{2}(r D_{h}^{2} \rho)    r A_{h} D_{h} \rho] | A_{h} v |^{2}  d t
	& \geq 
	- \mathbb{E} \int_{Q^*}  s \mathcal{O}_{\mu}((sh)^{2}) A_h(v^2)  d t
	 \geq 
	-  \mathbb{E} \int_{Q}  s \mathcal{O}_{\mu}((sh)^{2}) v^2 d t 
	.
\end{align*}
Similarity, the second term of the right hand of  \cref{eqI42} can be estimated  by
\begin{align*}
	 2 h^{2} \mathbb{E} \int_{Q} A_{h} D_{h} (r D_{h}^{2} \rho)    r A_{h} D_{h} \rho   |A_{h} D_{h} v|^{2} d t
	& \geq 
	- \mathbb{E} \int_{Q^*}  s \mathcal{O}_{\mu}((sh)^{2}) |D_hv|^2  d t
	.
\end{align*}
By \cref{eqdprpra},  and substituting the above two estimates into \cref{eqI42} yields
\begin{align}
	\label{eqI42a}
	2I_{42}\ge -\mathcal{A}_{1} - \mathcal{B}
	.
\end{align}

Combining \cref{eqI12,eqI22,eqI32,eqI42a}, for $ \lambda h (\delta T)^{-m} \leq 1$, we conclude that 
\begin{align}
    \label{eqIi2Final} \notag
    2\sum_{i=1}^{4} I_{i2} 
    & \geq 
    6 \mathbb{E} \int_{Q} s^{3} \mu^{4} \varphi^{3} |\partial_{x} \psi|^{4}  |v| ^{2}  d t 
    + 2 \mathbb{E} \int_{Q^{*}}  s \mu^{2} \varphi |\partial_{x} \psi|^{2}   |D_{h} v|^{2}  d t 
    - 2 \mathbb{E} \int_{\partial Q} s \mu \varphi  \partial_{x} \psi t_{r} ( |D_{h} v|^{2} ) \nu d t
    \\
    & \quad 
    + 2 \mathbb{E} \int_{Q} \bigg[
        \lambda \phi \theta_{t} \partial_{x}( r \partial_{x} \rho )
        - \frac{\theta_{t}}{\theta} s^{2} \mu^{2} \varphi^{2} |\partial_{x} \psi|^{2}
        + \frac{\theta_{t}}{\theta} s^{2} \mathcal{O}_{\mu}((s h)^{2})
    \bigg] v^{2} d t 
    -  \mathcal{A}_{1} - \mathcal{B}
    .
\end{align} 

\emph{Step 4.} We compute $ \sum\limits_{i=1}^{4} I_{i3} $.
By \cref{eqIij,eqA2u,eqdprpra,eqpr2pra,eqDIgbp,eqDuv,eqDu2}, letting $ q^{13} = r D_{h}^{2} \rho D_{h} ( r D_{h} \rho) $, for $ \lambda h (\delta T)^{-m} \leq 1$, we have 
\begin{align}
    \label{eqI13} \notag
    2I_{13} 
    & = 
    2\mathbb{E} \int_{Q} q^{13} A_{h}^{2} v v d t 
    \\ \notag
    & = 
    2 \mathbb{E} \int_{Q} q^{13} |v|^{2} d t 
    +  \frac{h^{2}}{2}\mathbb{E} \int_{Q} q^{13} D_{h}^{2} v  v d t 
    \\ \notag
    & = 
    2 \mathbb{E} \int_{Q} q^{13} |v|^{2} d t 
    -  \frac{h^{2}}{2}\mathbb{E} \int_{Q^{*}} A_{h} q^{13} |D_{h} v|^{2}   d t 
    +  \frac{h^{2}}{4}\mathbb{E} \int_{Q} D_{h}^{2} q^{13} |v|^{2}   d t 
    \\
    & \geq  
    - 2 \mathbb{E} \int_{Q} s^{3} \mu^{4} \varphi^{3} |\partial_{x} \psi|^{4}  |v| ^{2}  d t 
    - \mathcal{A}_{1}
    .
\end{align}

From \cref{eqIij,eqDIgbp,eqDuv,eqDu2,eqdprpra,eqprpr}, letting $ q^{23} = r A_{h}^{2} \rho D_{h} ( r D_{h} \rho) $, for $ \lambda h (\delta T)^{-m} \leq 1$, we see that 
\begin{align}
    \label{eqI23} \notag
    2I_{23} 
    & = 
    2 \mathbb{E} \int_{Q} q^{23} D_{h}^{2} v v d t 
    \\ \notag
    & = 
    - 2 \mathbb{E} \int_{Q^{*}} A_{h} q^{23} |D_{h} v|^{2}  d t 
    + \mathbb{E} \int_{Q} D_{h}^{2} q^{23} |v|^{2}  d t 
    \\
    & \geq  
    2 \mathbb{E} \int_{Q^{*}}  s \mu^{2} \varphi |\partial_{x} \psi|^{2} |D_{h} v|^{2}  d t 
    - \mathcal{A}_{1}
    .
\end{align}

Thanks to \cref{eqIij,eqdprpr}, we deduce that 
\begin{align}
    \label{eqI33} \notag
    2I_{33}
    & = 
    -2 \mathbb{E} \int_{Q} \lambda \phi \theta_{t} D_{h} (r D_{h} \rho) v^{2} d t 
    \\
    & \geq 
    -2 \mathbb{E} \int_{Q} \lambda \phi \theta_{t} \partial_{x}( r \partial_{x} \rho)  v^{2} d t 
    + \mathbb{E} \int_{Q} \frac{\theta_{t}}{\theta}   s^{2} \mathcal{O}_{\mu}((sh)^{2}) v^{2} d t 
    .
\end{align}

From Cauchy-Schwarz inequality and  \cref{eqIij},  we have 
\begin{align}
    \label{eqI43a} 
    2I_{43}
     = 
    2 \mathbb{E} \int_{Q} \Phi(v) \Psi(v) d t 
    \geq  
    - \mathbb{E} \int_{Q} |\Phi(v)|^{2}  d t 
    - \mathbb{E} \int_{Q} |\Psi(v)|^{2}  d t 
    .
\end{align}
By \cref{eqIij,eqDuv,eqAu2leq,eqAIgbp,eqdprpra}, for $ \lambda h (\delta T)^{-m} \leq 1$, a direct computation shows that
\begin{align}\label{ineqPhi}\notag
	- \mathbb{E} \int_{Q} |\Phi(v)|^{2}  d t 
 &\geq  
	- \mathbb{E} \int_{Q}  \mathcal{O}_{\mu}((sh)^{2})|A^2_hv|^{2} d t 
	-  \mathbb{E} \int_{Q}  \mathcal{O}_{\mu}((sh)^{2})|A_hD_hv|^{2} d t\\
	&\geq- \mathbb{E} \int_{Q}  \mathcal{O}_{\mu}((sh)^{2})|v|^{2} d t 
	-  \mathbb{E} \int_{Q^*}  \mathcal{O}_{\mu}((sh)^{2})|D_hv|^{2} d t
	,
\end{align}
and
\begin{align}\label{ineqPsi}
	&
	- \mathbb{E} \int_{Q} |\Psi(v)|^{2}  d t 
	\geq  
	- \mathbb{E} \int_{Q}  s^2\mathcal{O}_{\mu}(1)|v|^{2} d t 
	.
\end{align}
Substituting \cref{ineqPhi,ineqPsi} into \cref{eqI43a}, it follows that

\begin{align}
	\label{eqI43} 
	2I_{43}
	\geq
	- \mathcal{A}_{1}
	.
\end{align}

Hence, combining \cref{eqI13,eqI23,eqI33,eqI43}, for $ \lambda h (\delta T)^{-m} \leq 1$, we conclude that 
\begin{align}
    \label{eqIi3Final} \notag
    2\sum_{i=1}^{4} I_{i3} 
    & \geq 
    - 2 \mathbb{E} \int_{Q} s^{3} \mu^{4} \varphi^{3} |\partial_{x} \psi|^{4}  |v| ^{2}  d t 
    + 2 \mathbb{E} \int_{Q^{*}}  s \mu^{2} \varphi |\partial_{x} \psi|^{2} |D_{h} v|^{2}  d t
    \\
    & \quad 
    -2 \mathbb{E} \int_{Q} \lambda \phi \theta_{t} \partial_{x}( r \partial_{x} \rho)  v^{2} d t 
    + \mathbb{E} \int_{Q} \frac{\theta_{t}}{\theta}   s^{2} \mathcal{O}_{\mu}((sh)^{2}) v^{2} d t 
    - \mathcal{A}_{1}
    .
\end{align}

\emph{Step 5.} In this step, we summarize the $ I_{ij} $.
Combining \cref{eqIi1Final,eqIi2Final,eqIi3Final,eqDefIij}, for $ \mu \geq \mu_{1} $ and  $ \lambda h (\delta T)^{-m} \leq 1$, we obtain 
\begin{align}
    \label{eqSumIije1} \notag
    2 \mathbb{E} \int_{Q} I_{1} I_{2}
    & \geq 
    \mathbb{E} \int_{\mathcal{M}}  c\lambda^{2} \mu^{3} e^{2 \mu (6 m + 1)} |v(0)|^{2}
     +\mathbb{E}\int_{\mathcal{M}^{*}} | D_{h}  v(0)|^{2}+ \mathbb{E} \int_{Q^{*}} | D_{h}(d v ) |^{2} 
    \\ \notag
    & \quad 
    - \mathbb{E} \int_{Q}   \frac{\theta_{t}}{\theta}  ( 4 s^{2} \mu^{2} \varphi^{2} |\partial_{x} \psi|^{2} + s \mathcal{O}_{\mu}(1) + s^{2} \mathcal{O}_{\mu}((sh)^{2}))v^{2}  d t
    \\ \notag
    & \quad 
    + 4 \mathbb{E} \int_{Q} s^{3} \mu^{4} \varphi^{3} |\partial_{x} \psi|^{4}  |v| ^{2}  d t 
    + 4 \mathbb{E} \int_{Q^{*}}  s \mu^{2} \varphi |\partial_{x} \psi|^{2}   |D_{h} v|^{2}  d t 
    \\
    & \quad 
    - 2 \mathbb{E} \int_{\partial Q} s \mu \varphi  \partial_{x} \psi t_{r} ( |D_{h} v|^{2} ) \nu d t- \mathcal{A}_{1} -\mathcal{A}_{2} - \mathcal{A}_{3} -\mathcal{B}
    .
\end{align}

From \cref{eqWeightedFunctionTime,eqPsi}, noting $ |\theta_{t}| \leq C \theta^2   $ in $ (T/4, T) $, $\theta\in [1,2]$, $ |\theta_{t}| \leq  \lambda\mathcal{O}_{\mu}(1)   $ in $ (0, T/4) $, we have 
\begin{align}
    \label{eqSumIije2}
    & \mathbb{E} \int_{Q}   \frac{\theta_{t}}{\theta}  (  s \mathcal{O}_{\mu}(1) + s^{2} 
    \mathcal{O}_{\mu}((sh)^{2}))v^{2}  d t
     \leq 
    \mathbb{E} \int_{Q}  (  s^2 \mathcal{O}_{\mu}(1) + s^{3} \mathcal{O}_{\mu}((sh)^{2}))v^{2}  d t
     \leq 
     \mathcal{A}_{1}
    .
\end{align}
Further, noting $\theta_{t}<0$ in $ (0, T/4) $, we have
\begin{align}
    \label{eqSumIije3}
    & - 4 \mathbb{E} \int_{Q}   \frac{\theta_{t}}{\theta}     s^{2} \mu^{2} \varphi^{2} |\partial_{x} \psi|^{2} |v|^{2} d t 
     \geq 
    -\mathbb{E} \int_{Q} s^{3}\mu^2\varphi^2 \mathcal{O}(1) |v|^{2} d t \geq -\mathcal{A}_{1}
    .
\end{align}

Hence, from \cref{eqSumIije1,eqSumIije2,eqSumIije3,eqPCarS1e1} and noting \cref{ineqPhi,ineqPsi}, there exists $ \mu_{2} \geq \mu_{1} $ such that for all $ \mu \geq \mu_{2} $, one can find three constants $ 1 \geq \varepsilon_{1} > 0 $, $ h_{1} > 0 $ and $ \lambda_{1} \geq 1 $, so that for all $ \lambda \geq \lambda_{1} $, $ h \leq h_{1} $ and $ \lambda h (\delta T)^{-m} \leq \varepsilon_{1} $, we have 
\begin{align}
    \label{eqStep5} \notag
    &\mathbb{E} \int_{\mathcal{M}}  \lambda^{2} \mu^{3} e^{2 \mu (6 m + 1)} |v(0)|^{2}
    + \mathbb{E} \int_{Q} s^{3} \mu^{4} \varphi^{3}    |v| ^{2}  d t 
    +   \mathbb{E} \int_{Q^{*}}  s \mu^{2} \varphi    |D_{h} v|^{2}  d t
    + \mathbb{E} \int_{Q^{*}} | D_{h}(d v ) |^{2} \\\notag
    &\le
    C \mathbb{E} \int_{Q} | r f|^{2} d t  
    + C \mathbb{E} \int_{0}^{T} \int_{\mathcal{M}\cap G_{2}}  s^{3} \mu^{4} \varphi^{3}   |v| ^{2}  d t 
    + C \mathbb{E} \int_{0}^{T} \int_{\mathcal{M}^{*}\cap G_{2}}  s \mu^{2} \varphi     |D_{h} v|^{2}  d t\\
    &\quad 
    + C(\mu) \mathbb{E} \int_{Q} s^{2} |d v|^{2} 
    + C(\mu) h^{-2} \mathbb{E} \int_{\mathcal{M}} |v(T)|^{2}
    .
\end{align}

\emph{Step 6.} In this step, we return to the variable $ w $.

Noting $ w = \rho v $, combining \cref{eqDuv,eqdprpr,eqpr2pr}, for  $ \lambda h (\delta T)^{-m} \leq  1$, we obtain 
\begin{align}
	\notag
	\label{eqDwe1}
	& \mathbb{E} \int_{Q^{*}} s \mu^{2} \varphi r^{2}  |D_{h} w |^{2} d t 
	 =
	\mathbb{E} \int_{Q^{*}} s \mu^{2} \varphi r^{2}  |D_{h} (\rho v) |^{2} d t 
	\\ \notag
	& \leq 
	2 \mathbb{E} \int_{Q^{*}} s \mu^{2} \varphi r^{2}  ( |D_{h} \rho A_{h} v|^{2} + |A_{h} \rho  D_{h} v|^{2}) d t 
	\\ \notag
	& \leq 
	C \mathbb{E} \int_{Q^{*}} s^{3} \mu^{4} \varphi^{3}  |A_{h} v|^{2} d t 
	+ 2 \mathbb{E} \int_{Q^{*}} s \mu^{2} \varphi  |D_{h} v|^{2} d t 
	+ \mathbb{E} \int_{Q}   s^{3} \mathcal{O}_{\mu}((sh)^{2}) | v|^{2} d t 
	\\
	& \quad 
	+ \mathbb{E} \int_{Q^{*}} s \mathcal{O}_{\mu}((sh)^{2}) | D_{h} v|^{2} d t 
	.
\end{align}
From \cref{eqAu2leq,eqAIgbp,propADf}, we get 
\begin{align}
	\label{eqDwe2}
	\mathbb{E} \int_{Q^{*}} s^{3} \mu^{4} \varphi^{3}  |A_{h} v|^{2} d t 
	& \leq 
	\mathbb{E} \int_{Q} s^{3} \mu^{4} \varphi^{3}  |  v|^{2} d t 
	+ \mathbb{E} \int_{Q} s^{3} \mathcal{O}_{\mu}((sh)^{2}) |  v|^{2} d t 
	.
\end{align}
Combining \cref{eqDwe1,eqDwe2}, for $ \lambda h (\delta T)^{-m} \leq  1$, we have 
\begin{align}
	\label{eqDw}
	\mathbb{E} \int_{Q^{*}} s \mu^{2} \varphi r^{2}  |D_{h} w |^{2} d t 
	\leq 
	C \mathbb{E} \int_{Q} s^{3} \mu^{4} \varphi^{3}  |  v|^{2} d t 
	+ 2 \mathbb{E} \int_{Q^{*}} s \mu^{2} \varphi  |D_{h} v|^{2} d t 
	+ \mathcal{A}_{1}
	.
\end{align}

From \cref{eqDuv,eqdprpr}, for  $ \lambda h (\delta T)^{-m} \leq  1$, we obtain
\begin{align*}
	 r D_{h} w =  r D_{h} \rho A_{h} v  + r A_{h} \rho D_{h} v 
	 = r D_{h} \rho A_{h} v  + D_{h} v + \mathcal{O}_{\mu} ((sh)^{2}) D_{h} v 
	 ,
\end{align*}
which implies
\begin{align}
	\notag
	\label{eqRwLe2}
	& \mathbb{E} \int_{0}^{T} \int_{M^{*} \cap G_{2}} s \mu^{2} \varphi |D_{h} v|^{2} d t 
	\\ \notag
	& \leq 
	C \mathbb{E} \int_{0}^{T} \int_{M^{*} \cap G_{2}} s \mu^{2} \varphi (r D_{h} \rho)^{2}|A_{h} v|^{2} d t 
	+ C  \mathbb{E} \int_{0}^{T} \int_{M^{*} \cap G_{2}} s \mu^{2} \varphi r^{2} |D_{h} w|^{2} d t 
	\\ \notag
	& \quad 
	+ \mathbb{E} \int_{Q^{*}} s \mathcal{O}_{\mu}((sh)^{2}) |D_{h} v|^{2} d t
	\\ \notag
	& \leq 
	C \mathbb{E} \int_{0}^{T} \int_{M  \cap G_{1}} s^{3} \mu^{4} \varphi^{3} r^{2} | w|^{2} d t 
	+ C  \mathbb{E} \int_{0}^{T} \int_{M^{*} \cap G_{1}} s \mu^{2} \varphi r^{2} |D_{h} w|^{2} d t 
	\\ \notag
	& \quad 
	+ \mathbb{E} \int_{Q^{*}} s \mathcal{O}_{\mu}((sh)^{2}) |D_{h} v|^{2} d t
	+ \mathbb{E} \int_{Q } s^{3} \mathcal{O}_{\mu}((sh)^{2}) | v|^{2} d t
	\\
	& \leq 
	C \mathbb{E} \int_{0}^{T} \int_{M  \cap G_{1}} s^{3} \mu^{4} \varphi^{3} r^{2} | w|^{2} d t 
	+ C  \mathbb{E} \int_{0}^{T} \int_{M^{*} \cap G_{1}} s \mu^{2} \varphi r^{2} |D_{h} w|^{2} d t  
	+ \mathcal{A}_{1} 
	, 
\end{align}
where we have used $ |A_{h} v|^{2} \leq C (|\tau_{+}v|^{2} + |\tau_{-} v|^{2}) $.

Combining \cref{eqStep5,eqDw,eqRwLe2}, there exists $ \mu_{3} \geq \mu_{2} $ such that for all $ \mu \geq \mu_{3} $, one can find three positive constants $ \varepsilon_{2} \leq \varepsilon_{1}   $, $h_{2}  \leq  h_{1}   $ and $ \lambda_{2} \geq \lambda_{1} $, so that for all $ \lambda \geq \lambda_{2} $, $ h \leq h_{2} $ and $ \lambda h (\delta T)^{-m} \leq \varepsilon_{2} $, we have 
\begin{align}
    \label{eqStep6} \notag
    &\mathbb{E} \int_{\mathcal{M}}  \lambda^{2} \mu^{3} e^{2 \mu (6 m + 1)} e^{2 s \phi}  |w|^{2} \Big|_{t=0}
    + \mathbb{E} \int_{Q} s^{3} \mu^{4} \varphi^{3}  e^{2 s \phi}   |w| ^{2}  d t 
    +   \mathbb{E} \int_{Q^{*}}  s \mu^{2} \varphi   e^{2 s \phi}   |D_{h} w|^{2}  d t\\\notag
    &\le
    C \mathbb{E} \int_{Q} e^{2 s \phi} |  f|^{2} d t 
    + C \mathbb{E} \int_{0}^{T} \int_{\mathcal{M}\cap G_{1}}  s^{3} \mu^{4} \varphi^{3}  e^{2 s \phi}  |w| ^{2}  d t 
    + C \mathbb{E} \int_{0}^{T} \int_{\mathcal{M}^{*}\cap G_{1}}  s \mu^{2} \varphi  e^{2 s \phi}    |D_{h} w|^{2}  d t \\
    &\quad+ C(\mu) \mathbb{E} \int_{Q} s^{2} e^{2 s \phi}  |g|^{2} 
    + C(\mu) h^{-2} \mathbb{E} \int_{\mathcal{M}} e^{2 s \phi}  |w|^{2} \Big|_{t=T}
    .
\end{align}

\emph{Step 7.} In this step, we estimate the local terms. 

Choose a function $ \xi \in C_{0}^{\infty}(G_{0};[0,1]) $ such that $ \xi =1 $ in $ G_{1} $.
Thanks to It\^o's formula and \cref{eqW}, we have 
\begin{align*}
	d ( s \varphi \xi^{2} e^{2 s \phi} w^{2}) 
	& = 
    2 s \varphi \xi^{2} e^{2 s \phi} w d w 
    + \varphi \xi^{2} \partial_{t} ( s  e^{2 s \phi} ) w^{2} dt 
    + s \varphi \xi^{2} e^{2 s \phi} (d w)^{2}
	.
\end{align*}
Hence, we have 
\begin{align} 
	\label{eqLocalD2we1}
	\notag
	& \mathbb{E} \int_{\mathcal{M} \cap G_{0}} s \varphi \xi^{2} e^{2 s \phi} |w|^{2} \Big|_{t = T}
	- \mathbb{E} \int_{\mathcal{M} \cap G_{0}} s \varphi \xi^{2} e^{2 s \phi} |w|^{2} \Big|_{t = 0}
	\\  \notag
	& =
	\mathbb{E} \int_{0}^{T} \int_{\mathcal{M} \cap G_{0}} \varphi \xi^{2} \partial_{t} ( s  e^{2 s \phi} ) w^{2} dt  
	+ 2 \mathbb{E} \int_{0}^{T} \int_{\mathcal{M} \cap G_{0}} s \varphi \xi^{2} e^{2 s \phi} w (- D_{h}^{2} w + f) d t   
	\\
	& \quad 
	+  \mathbb{E} \int_{0}^{T} \int_{\mathcal{M} \cap G_{0}} s \varphi \xi^{2} e^{2 s \phi} |g|^{2} d t 
	.
\end{align}
From \cref{eqLocalD2we1,eqDIgbp,eqDuv}, a direct computation shows that
\begin{align}
    \label{eqStep7E1} \notag
    & 
    \mathbb{E} \int_{0}^{T} \int_{\mathcal{M} \cap G_{0}} \varphi \xi^{2} \partial_{t} ( s  e^{2 s \phi} ) w^{2} dt  
    + 
    \mathbb{E} \int_{0}^{T} \int_{\mathcal{M}^{*} \cap G_{0}} 2 s A_{h} (\xi^{2} \varphi e^{2 s \phi}) |D_{h} w |^{2} d t 
    \\ \notag
    & =
    - 2\mathbb{E} \int_{0}^{T} \int_{\mathcal{M} \cap G_{0}}s \varphi \xi^{2} e^{2 s \phi} w f d t - \mathbb{E} \int_{0}^{T} \int_{\mathcal{M} \cap G_{0}} s \varphi \xi^{2} e^{2 s \phi} |g|^{2} d t 
    + \mathbb{E} \int_{\mathcal{M} \cap G_{0}} s \varphi \xi^{2} e^{2 s \phi} |w|^{2} \Big|_{t = T}
    \\
    & \quad 
    - \mathbb{E} \int_{\mathcal{M} \cap G_{0}} s \varphi \xi^{2} e^{2 s \phi} |w|^{2} \Big|_{t = 0}-  \mathbb{E} \int_{0}^{T} \int_{\mathcal{M}^{*} \cap G_{0}} 2 s D_{h} (\xi^{2} \varphi e^{2 s \phi}) A_{h} w D_{h} w  d t 
\end{align}

Thanks to \cref{eqWeightFuncitonDefine,eqWeightedFunctionTime}, noting that $ \theta_{t} < 0 $ in $ (0, T/4) $, $ \theta_{t} < C\theta^2 $ in $ (T/4, T) $ and $\phi<0$, there exist $ \mu_{4} \geq 1 $ and $ \lambda_{3} \geq 1 $ such that for all $ \mu \geq \mu_{4} $ and $ \lambda \geq \lambda_{3} $, we have 
\begin{align}
    \label{eqStep7E2a} \notag
    & 
    \mathbb{E} \int_{0}^{T} \int_{\mathcal{M} \cap G_{0}} \varphi \xi^{2} \partial_{t} ( s  e^{2 s \phi} ) |w|^{2} dt  
    \\ \notag
    & =
    \mathbb{E} \int_{0}^{T} \int_{\mathcal{M} \cap G_{0}} \varphi \xi^{2}  ( \lambda \theta_{t} + 2\lambda s \theta_{t} \phi   ) e^{2 s \phi}|w|^{2} dt  
    \\ 
    & \geq 
    - \mathbb{E} \int_{\frac{T}{4}}^{T} \int_{\mathcal{M} \cap G_{0}} [s^{2} \mathcal{O}_{\mu}(1)+ \lambda^{-1} s^{3} \mathcal{O}_{\mu}(1)] e^{2 s \phi} |w|^{2} d t 
    + \mathbb{E} \int_{0}^{\frac{T}{4}} \int_{\mathcal{M} \cap G_{0}} \lambda \theta_{t}\varphi \xi^{2}   e^{2 s \phi}|w|^{2}d t 
    .
\end{align}
Noting $|\theta_{t}|\le C\lambda\mu^2 e^{\mu(6m-4)}\le C\lambda\mu^2 \varphi$, then from \cref{eqStep7E2a} we have
\begin{align}
	\label{eqStep7E2} \notag
	& 
	\mathbb{E} \int_{0}^{T} \int_{\mathcal{M} \cap G_{0}} \varphi \xi^{2} \partial_{t} ( s  e^{2 s \phi} ) |w|^{2} dt \\ \notag
	& \geq 
	- \mathbb{E} \int_{\frac{T}{4}}^{T} \int_{\mathcal{M} \cap G_{0}} [s^{2} \mathcal{O}_{\mu}(1)+ \lambda^{-1} s^{3} \mathcal{O}_{\mu}(1)] e^{2 s \phi} |w|^{2} d t 
	- \mathbb{E} \int_{0}^{\frac{T}{4}} \int_{\mathcal{M} \cap G_{0}} s^{2}\mu^2\varphi^2 \mathcal{O}(1)   e^{2 s \phi}|w|^{2}d t \\
	&\geq - \mathbb{E} \int_{Q} [s^{2} \mathcal{O}_{\mu}(1)+ \lambda^{-1} s^{3} \mathcal{O}_{\mu}(1)] e^{2 s \phi} |w|^{2} d t 
	.
\end{align}    
   
By \cref{propADf}, it holds that 
\begin{align}
    \label{eqStep7E3} \notag
    & 
    \mathbb{E} \int_{0}^{T} \int_{\mathcal{M}^{*} \cap G_{0}} 2 s A_{h} (\xi^{2} \varphi e^{2 s \phi}) |D_{h} w |^{2} d t 
    \\
    & \geq 
    2 \mathbb{E} \int_{0}^{T} \int_{\mathcal{M}^{*} \cap G_{0}}  s \varphi \xi^{2} e^{2 s \phi}  |D_{h} w |^{2} d t 
    - C \mathbb{E} \int_{0}^{T} \int_{\mathcal{M}^{*} \cap G_{0}}  s \mathcal{O}_{\mu}((s h)^{2}) e^{2 s \phi}  |D_{h} w |^{2} d t 
    .
\end{align}

Thanks to Cauchy-Schwarz inequality, we have 
\begin{align}
    \label{eqStep7E4}  
    - 2\mathbb{E} \int_{0}^{T} \int_{\mathcal{M} \cap G_{0}}s \varphi \xi^{2} e^{2 s \phi} w f d t 
    \leq 
    C \mathbb{E} \int_{0}^{T} \int_{\mathcal{M} \cap G_{0}} s^{2} \mu^{2} \varphi^{2} e^{2 s \phi} |w|^{2} d t 
    + C \mathbb{E} \int_{Q} \mu^{-2} e^{2 s \phi} |f|^{2} d t 
    .
\end{align}
For $ \lambda h (\delta T)^{-m} \leq 1$, we obtain 
\begin{align}
    \label{eqStep7E5}  
    \mathbb{E} \int_{\mathcal{M} \cap G_{0}} s \varphi \xi^{2} e^{2 s \phi} |w|^{2} \Big|_{t = T}
    \leq 
    C(\mu) h^{-1} \mathbb{E} \int_{\mathcal{M} \cap G_{0}}  e^{2 s \phi} |w|^{2} \Big|_{t = T}
    .
\end{align} 
By \cref{eqDu2,eqDIgbp,propADf}, we get 
\begin{align}
    \label{eqStep7E6} \notag 
    &
    -  \mathbb{E} \int_{0}^{T} \int_{\mathcal{M}^{*} \cap G_{0}} 2 s D_{h} (\xi^{2} \varphi e^{2 s \phi}) A_{h} w D_{h} w  d t 
    \\ \notag
    & =
    \mathbb{E} \int_{0}^{T} \int_{\mathcal{M}\cap G_{0}}  s D_{h}^{2} (\xi^{2} \varphi e^{2 s \phi})   |w|^{2}  d t 
    \\
    & \leq 
    C \mathbb{E} \int_{0}^{T} \int_{\mathcal{M}\cap G_{0}}  s^{3} \mu^{2} \varphi^{3}  e^{2 s \phi}    |w|^{2}  d t 
    + C \mathbb{E}  \int_{0}^{T} \int_{\mathcal{M}\cap G_{0}} s^{3} \mathcal{O}_{\mu}((s h)^{2}) e^{2 s \phi} |w|^{2} d t
    .
\end{align}

Hence, combining \cref{eqStep7E1,eqStep7E2,eqStep7E3,eqStep7E4,eqStep7E5,eqStep7E6}, there exists $ \mu_{5} \geq \max\{ \mu_{3}, \mu_{4} \} $ such that for all $ \mu \geq \mu_{5} $, one can find three positive constants $ \varepsilon_{3} \leq \varepsilon_{2}   $, $h_{3}  \leq  h_{2}   $ and $ \lambda_{4} \geq \max\{ \lambda_{2}, \lambda_{3} \} $, so that for all $ \lambda \geq \lambda_{4} $, $ h \leq h_{3} $ and $ \lambda h (\delta T)^{-m} \leq \varepsilon_{3} $, we have 
\begin{align*}
    & 
     \mathbb{E} \int_{0}^{T} \int_{\mathcal{M}^{*} \cap G_{1}}  s \mu^{2} \varphi  e^{2 s \phi}  |D_{h} w |^{2} d t 
    \\
    & \leq 
    C \mathbb{E} \int_{0}^{T} \int_{\mathcal{M}\cap G_{0}}  s^{3} \mu^{4} \varphi^{3}  e^{2 s \phi}    |w|^{2}  d t 
    + C(\mu) h^{-1} \mathbb{E} \int_{\mathcal{M} \cap G_{0}}  e^{2 s \phi} |w|^{2} \Big|_{t = T}
    + C \mathbb{E} \int_{Q}  e^{2 s \phi} |f|^{2} d t\\
    &\quad+\mathbb{E} \int_{Q} [s^{2} \mathcal{O}_{\mu}(1)+ \lambda^{-1} s^{3} \mathcal{O}_{\mu}(1)] e^{2 s \phi} |w|^{2} d t 
    .
\end{align*}
Combining it with \cref{eqStep6}, we complete the proof.
\end{proof}

\section{The controllability for  linear stochastic semi-discrete parabolic equations }

Consider the following linear forward stochastic semi-discrete parabolic equation, which is the special case for \cref{eqDiscreteWithoutDiffuse}:
\begin{align}
    \label{eqDiscreteLinear}
    \begin{cases}
        d y - D_{h}^{2} y  d t  = (F + \chi_{G_{0}} u ) d t + v d W(t) & \text { in } Q,\\ 
        y(t,0) = y(t,1) = 0 & \text { on } (0,T), \\ 
        y(0, x) = y_{0} & \text { in } \mathcal{M}
        .
    \end{cases}
\end{align}

Recalling \cref{eqWeightFuncitonDefine}, we fix parameters $ \lambda, \mu $ and $ h $ such that the inequality \cref{eqDisCarleman} holds true.
We define the space
\begin{align}
    \label{eqSlambdamu}
    \mathfrak{S}_{\lambda, \mu,h} = \Big\{  
        F \in L^{2}_{\mathbb{F}}(0,T; L^{2}_{h}(\mathcal{M})) \mid 
        \mathbb{E} \int_{Q} s^{-3} e^{- 2 s \phi} |F|^{2} d t< \infty
    \Big\}
    ,
\end{align}
endowed with the norm 
\begin{align}
    \label{eqSlambdamuNorm}
    | F |_{\mathfrak{S}_{\lambda, \mu,h}}^{2} = \mathbb{E} \int_{Q} s^{-3} e^{- 2 s \phi} |F|^{2} d t
    .
\end{align}
Here, although the space $\mathfrak{S}_{\lambda, \mu,h}$ does not explicitly contain $h$, both $\lambda$ and $\mu$ have a certain matching relationship with $h$, making it related to $h$ as well.

We have the following null controllability for \cref{eqDiscreteLinear}.
\begin{theorem}
    For all $ \mu, \lambda  $ and $ h $ such that \cref{eqDisCarleman} holds, there exists a positive constant    $ C=C(\mu) $
    such that  for any   $ y_{0} \in L^{2}_{\mathcal{F}_{0}}(\Omega;  L^{2}_{h} (\mathcal{M})) $ and any source term $ F \in \mathfrak{S}_{\lambda, \mu,h} $,
    there exist $ (u, v) \in  L^{2}_{\mathbb{F}}(0,T;   L^{2}_{h}(G_{0} \cap \mathcal{M})) \times L^{2}_{\mathbb{F}}(0,T; L^{2}_{h}( \mathcal{M}))$, so that the solution to \cref{eqDiscreteLinear} satisfies 
    \begin{align}
        \label{eqLinearYT}
        \mathbb{E} \int_{\mathcal{M}} |y(T)|^{2} 
        \leq 
        C \mathcal{E}_{\lambda,\mu,h}  \Big(
            \mathbb{E} \int_{\mathcal{M}} s^{-2} e^{- 2 s \phi} |y|^{2} \Big|_{t = 0} 
            + \mathbb{E} \int_{Q} s^{-3} e^{- 2 s \phi} |F|^{2} d t 
        \Big) 
        ,
    \end{align}
    where 
    \begin{align}
        \label{eqEpsilon0}
        \mathcal{E}_{\lambda,\mu,h} = C h^{-2} e^{- \lambda \theta(T) \mu e^{6 \mu (m+1)}}
        .
    \end{align}
    Moreover, the following estimate holds:
    \begin{align}
        \label{eqLinearYUV} \notag
        & 
        \mathbb{E} \int_{Q} e^{- 2 s \phi} |y|^{2} dt 
        + \mathbb{E} \int_{0}^{T} \int_{\mathcal{M} \cap G_{0}} s^{-3} e^{- 2 s \phi} |u|^{2} dt 
        + \mathbb{E} \int_{Q} s^{-2} e^{- 2 s \phi} |v|^{2} dt 
        \\
        & \leq 
        C \Big(
            \mathbb{E} \int_{\mathcal{M}} s^{-2} e^{- 2 s \phi} |y|^{2} \Big|_{t = 0}
            +  \mathbb{E} \int_{Q} s^{-3} e^{- 2 s \phi} |F|^{2} d t 
        \Big)
        .
    \end{align} 
\end{theorem}

\begin{proof}

We apply the penalized Hilbert Uniqueness Method.

\emph{Step 1.} 
For any $ \varepsilon > 0 $, we introduce the functional
\begin{align}
    \label{eqLinearE1} \notag
    J_{\varepsilon} (u ,v) 
    & = 
    \frac{1}{2} \left(\mathbb{E} \int_{Q} e^{- 2 s \phi} |y|^{2} d t 
    +  \mathbb{E} \int_{0}^{T} \int_{\mathcal{M} \cap G_{0}} s^{-3} e^{- 2 s \phi} |u|^{2} dt 
    +  \mathbb{E} \int_{Q} s^{-2} e^{- 2 s \phi} |v|^{2} dt\right)
    \\
    & \quad 
    + \frac{1}{\varepsilon} \mathbb{E} \int_{\mathcal{M}} |y(T)|^{2}
\end{align}
and the optimal control problem:
\begin{align}
    \label{eqLinearE2}
    \left\{
        \begin{aligned}
           & \min\limits_{(u, v) \in \mathcal{H} } J_{\varepsilon}(u,v) \\
         &\text{subject to equation \cref{eqDiscreteLinear}}, 
        \end{aligned}
    \right.
\end{align}
where 
\begin{align*}
    \mathcal{H} = \Big\{ 
        (u, v) \in L^{2}_{\mathbb{F}}(0,T;   L^{2}_{h}(G_{0} \cap \mathcal{M})) \times L^{2}_{\mathbb{F}}(0,T; L^{2}_{h}( \mathcal{M})) 
        ~ \mid ~ 
        & \mathbb{E} \int_{0}^{T} \int_{\mathcal{M} \cap G_{0}} s^{-3} e^{- 2 s \phi} |u|^{2} dt < \infty,
        \\
        & \mathbb{E} \int_{Q} s^{-2} e^{- 2 s \phi} |v|^{2} dt < \infty
    \Big\}.
\end{align*}

It is easily seen that for any $ \varepsilon > 0 $, the problem \cref{eqLinearE2} admits a unique optimal solution $ (u_{\varepsilon}, v_{\varepsilon}) $.
By the standard variational method (see \cite[Chapter 3, Section 2]{Lions}), the pair $ (u_{\varepsilon}, v_{\varepsilon}) $ can be characterized as 
\begin{align}
    \label{eqLinearE3}
    u_{\varepsilon} = - \chi_{G_{0}} s^{3} e^{2 s \phi} z_{\varepsilon}, \quad \quad 
    v_{\varepsilon} = - s^{2} e^{2 s \phi} Z_{\varepsilon} 
    \quad \text{ in } Q, \ \mathbb{P}\text{-a.s.,}
\end{align}
where $ (z_{\varepsilon}, Z_{\varepsilon}) $ solves 
\begin{align}
    \label{eqLinearE5}
    \left\{
        \begin{aligned}
            & d z_{\varepsilon} + D_{h}^{2} z_{\varepsilon} d t = - e^{- 2 s \phi} y_{\varepsilon} d t + Z_{\varepsilon} d W(t) && \text { in } Q,\\ 
            & z_{\varepsilon}(t,0) = z_{\varepsilon}(t,1) = 0 && \text { on } (0,T), \\ 
            & z_{\varepsilon}(T, x) = \dfrac{1}{\varepsilon} y_{\varepsilon}(T, x) && \text { in } \mathcal{M},
        \end{aligned}
    \right.
\end{align}
and where $ y_{\varepsilon} $ is the solution of \cref{eqDiscreteLinear} with the controls $ (u, v) = (u_{\varepsilon}, v_{\varepsilon}) $.

\emph{Step 2.}
Thanks to It\^o's formula, \cref{eqLinearE5,eqDiscreteLinear}, we deduce 
\begin{align}
    \label{eqLinearE4} \notag
    \mathbb{E} \int_{\mathcal{M}} y_{\varepsilon}(T) z_{\varepsilon}(T) 
    & = 
    \mathbb{E} \int_{\mathcal{M}} y_{\varepsilon}(0) z_{\varepsilon}(0) 
    + \mathbb{E} \int_{Q} y_{\varepsilon} (- D_{h}^{2} z_{\varepsilon} - e^{- 2 s \phi} y_{\varepsilon}) dt
    \\
    & \quad 
    + \mathbb{E} \int_{Q} z_{\varepsilon} (D_{h}^{2} y_{\varepsilon} + F + \chi_{G_{0}} u_{\varepsilon}) dt
    + \mathbb{E} \int_{Q} v_{\varepsilon} Z_{\varepsilon} dt
    .
\end{align}
Combining \cref{eqLinearE4,eqLinearE3,eqDIgbp}, we have 
\begin{align}
    \label{eqLinearE6} \notag
    & \mathbb{E} \int_{Q}   e^{- 2 s \phi} |y_{\varepsilon}|^{2} dt
    + \mathbb{E} \int_{0}^{T} \int_{\mathcal{M} \cap G_{0}} s^{3} e^{2 s \phi} |z_{\varepsilon}|^{2} d t 
    + \mathbb{E} \int_{Q}   s^{2} e^{2 s \phi} |Z_{\varepsilon}|^{2} d t 
    + \frac{1}{\varepsilon} \mathbb{E} \int_{\mathcal{M}}    |y_{\varepsilon}(T)|^{2} 
    \\
    & = 
    \mathbb{E} \int_{\mathcal{M}} y_{0} z_{\varepsilon}(0) 
    + \mathbb{E} \int_{Q} z_{\varepsilon} F d t 
    .
\end{align}
Applying Carleman estimate \cref{thmCarlemanEstimate} with $ w = z_{\varepsilon}$, $ f = - e^{-2 s \phi} y_{\varepsilon} $ and $ g = Z_{\varepsilon} $, we obtain 
\begin{align}
    \label{eqLinearE7} \notag
    & 
    \mathbb{E} \int_{\mathcal{M}} s^{2}(0) e^{2 s (0) \phi} |z_{\varepsilon}(0)|^{2} 
    + \mathbb{E} \int_{Q} s^{3} e^{2 s \phi } |z_{\varepsilon}|^{2} d t 
    \\ \notag
    & \leq C \Big(
        \mathbb{E} \int_{0}^{T}\int_{G_{0} \cap \mathcal{M}} s^{3} e^{2 s \phi} |z_{\varepsilon}|^{2} d t
        + \mathbb{E} \int_{Q} e^{- 2 s \phi} |y_{\varepsilon}|^{2} d t
        + \mathbb{E} \int_{Q} s^{2} e^{2 s \phi} |Z_{\varepsilon}|^{2} d t
    \\
    & \quad \quad ~ 
    + h^{-2} \varepsilon^{-2} \mathbb{E} \int_{\mathcal{M}} e^{2 s(T) \phi}|y_{\varepsilon} (T)|^2  \Big).
\end{align}

From Cauchy-Schwarz inequality,  for any $ \kappa > 0 $, we have 
\begin{align}
    \label{eqLinearE8} \notag
    &
    \mathbb{E} \int_{\mathcal{M}} y_{0} z_{\varepsilon}(0) 
    + \mathbb{E} \int_{Q} z_{\varepsilon} F d t 
    \\ \notag
    & \leq 
    \kappa \Big(
        \mathbb{E} \int_{\mathcal{M}} s^{2}(0)  e^{2 s (0) \phi} |z_{\varepsilon}(0)|^{2} 
    + \mathbb{E} \int_{Q} s^{3} e^{2 s \phi } |z_{\varepsilon}|^{2} d t 
    \Big)
    \\
    & \quad 
    + C(\kappa) \Big(
        \mathbb{E} \int_{\mathcal{M}} s^{-2}(0)  e^{- 2 s (0) \phi} |y_{0}|^{2} 
        + \mathbb{E} \int_{Q} s^{- 3} e^{- 2 s \phi } |F|^{2} d t 
    \Big)
    .
\end{align}
Letting $ \kappa > 0 $ be small enough, from \cref{eqLinearE8,eqLinearE7,eqLinearE6}, we conclude that 
\begin{align*}
    & \mathbb{E} \int_{Q}   e^{- 2 s \phi} |y_{\varepsilon}|^{2} dt
    + \mathbb{E} \int_{0}^{T} \int_{\mathcal{M} \cap G_{0}} s^{3} e^{2 s \phi} |z_{\varepsilon}|^{2} d t 
    + \mathbb{E} \int_{Q}   s^{2} e^{2 s \phi} |Z_{\varepsilon}|^{2} d t 
    + \frac{1}{\varepsilon} \mathbb{E} \int_{\mathcal{M}}    |y_{\varepsilon}(T)|^{2} 
    \\
    & \le 
    C \Big(
        \mathbb{E} \int_{\mathcal{M}} s^{-2}(0)  e^{- 2 s (0) \phi} |y_{0}|^{2} 
        + \mathbb{E} \int_{Q} s^{- 3} e^{- 2 s \phi } |F|^{2} d t 
    \Big)
    + C h^{-2} \varepsilon^{-2} \mathbb{E} \int_{\mathcal{M}} e^{2 s(T) \phi}|y_{\varepsilon} (T)|^2 
    .
\end{align*}
By the above inequality and \cref{eqLinearE3}, it holds that
\begin{align}
    \label{eqLinearE9} \notag
    & \mathbb{E} \int_{Q}   e^{- 2 s \phi} |y_{\varepsilon}|^{2} dt
    + \mathbb{E} \int_{0}^{T} \int_{\mathcal{M} \cap G_{0}} s^{- 3} e^{- 2 s \phi} |u_{\varepsilon}|^{2} d t 
    + \mathbb{E} \int_{Q}   s^{- 2} e^{- 2 s \phi} |v_{\varepsilon}|^{2} d t 
    + \frac{1}{\varepsilon} \mathbb{E} \int_{\mathcal{M}}    |y_{\varepsilon}(T)|^{2} 
    \\
    & \le
    C \Big(
        \mathbb{E} \int_{\mathcal{M}}  s^{-2}(0) e^{- 2 s (0) \phi} |y_{0}|^{2} 
        + \mathbb{E} \int_{Q} s^{- 3} e^{- 2 s \phi } |F|^{2} d t 
    \Big)
    + C h^{-2} \varepsilon^{-2} \mathbb{E} \int_{\mathcal{M}} e^{2 s(T) \phi}|y_{\varepsilon} (T)|^2 
    .
\end{align}

From \cref{eqWeightFuncitonDefine,eqEpsilon0}, for $ \mu > 2 $, we have 
\begin{align*}
    e^{2 s(T) \phi} \leq e^{ - \lambda \theta(T) \mu e^{6 \mu (m+1)} },
\end{align*}
which yields that, for $ \varepsilon \geq \mathcal{E}_{\lambda, \mu, h} $,
\begin{align}
    \label{eqLinearE10} \notag
    C h^{-2} \varepsilon^{-2} \mathbb{E} \int_{\mathcal{M}} e^{2 s(T) \phi}|y_{\varepsilon} (T)|^2 
    & \leq 
    C h^{-2} \varepsilon^{-2}  e^{ - \lambda \theta(T) \mu e^{6 \mu (m+1)} } \mathbb{E} \int_{\mathcal{M}} |y_{\varepsilon} (T)|^2 
    \\
    & \leq  
    \frac{1}{2 \varepsilon} \mathbb{E} \int_{\mathcal{M}} |y_{\varepsilon} (T)|^2 
    .
\end{align}
By \cref{eqLinearE9,eqLinearE10}, for $ \varepsilon \geq \mathcal{E}_{\lambda, \mu, h} $, we see that 
\begin{align}
    \label{eqLinearE11} \notag
    & \mathbb{E} \int_{Q}   e^{- 2 s \phi} |y_{\varepsilon}|^{2} dt
    + \mathbb{E} \int_{0}^{T} \int_{\mathcal{M} \cap G_{0}} s^{- 3} e^{- 2 s \phi} |u_{\varepsilon}|^{2} d t 
    + \mathbb{E} \int_{Q}   s^{- 2} e^{- 2 s \phi} |v_{\varepsilon}|^{2} d t 
    + \frac{1}{\varepsilon} \mathbb{E} \int_{\mathcal{M}}    |y_{\varepsilon}(T)|^{2} 
    \\
    & \leq 
    C  \Big(
         \mathbb{E} \int_{\mathcal{M}} s^{-2}(0) e^{- 2 s (0) \phi} |y_{0}|^{2} 
        + \mathbb{E} \int_{Q} s^{- 3} e^{- 2 s \phi } |F|^{2} d t 
    \Big)
    .
\end{align}

\emph{Step 3.} Let $ \varepsilon_{k} = \mathcal{E}_{\lambda, \mu, h} + \frac{1}{k} $, where $ k \in \mathbb{N}^{+} $. 
Since the right-hand side of \cref{eqLinearE11} is uniform with respect to $ \varepsilon $, there exists $ (\hat{u}, \hat{v}, \hat{y}, \hat{y}_{T}) $ such that 
\begin{align}
    \label{eqLinearE13}
    \begin{cases}
        u_{\varepsilon_{k}}   \rightharpoonup  \hat{u}  & \text{ weakly in  } L^{2}_{\mathbb{F}}(0,T; L^{2}_{h}( G_{0}\cap \mathcal{M})),
        \\
        v_{\varepsilon_{k}}   \rightharpoonup  \hat{v}  & \text{ weakly in  } L^{2}_{\mathbb{F}}(0,T; L^{2}_{h}( \mathcal{M})),
        \\
        y_{\varepsilon_{k}}   \rightharpoonup  \hat{y}  & \text{ weakly in  } L^{2}_{\mathbb{F}}(0,T; L^{2}_{h}( \mathcal{M})),
        \\
        y_{\varepsilon_{k}}(T)   \rightharpoonup  \hat{y}_{T}  & \text{ weakly in  } L^{2}_{\mathcal{F}_{T}}(\Omega; L^{2}_{h}( \mathcal{M})).
    \end{cases}
\end{align}

We claim that $ \hat{y} $ is the solution to \cref{eqDiscreteLinear} associated to $ (\hat{u}, \hat{v}) $ and $ \hat{y}(T) = \hat{y}_{T} $.
In fact, let $ \tilde{y} $ be the solution to \cref{eqDiscreteLinear} with controls $ (\hat{u}, \hat{v}) $. 
For any $ p \in L^{2}_{\mathbb{F}}(0,T; L^{2}_{h}( \mathcal{M})) $ and $ q \in L^{2}_{\mathcal{F}_{T}}(\Omega; L^{2}_{h}( \mathcal{M})) $, consider $ (z, Z) $ solves 
\begin{align}
    \label{eqLinearE12}
    \begin{cases}
        d z + D_{h}^{2} z  d t  = p d t + Z d W(t) & \text { in } Q,\\ 
        z(t,0) = z(t,1) = 0 & \text { on } (0,T), \\ 
        z(T, x) = q(x) & \text { in } \mathcal{M}
        .
    \end{cases}
\end{align}
Thanks to It\^o's formula and  \cref{eqDiscreteLinear}, we obtain 
\begin{align*}
    \mathbb{E} \int_{\mathcal{M}} y_{\varepsilon_{k}}(T) q 
    - \mathbb{E} \int_{\mathcal{M}} y_{0} z(0)
    =
     \mathbb{E} \int_{Q} p y_{\varepsilon_{k}} d t 
    + \mathbb{E} \int_{Q} F z d t 
    + \mathbb{E} \int_{0}^{T} \int_{\mathcal{M} \cap G_{0}} u_{\varepsilon_{k}} z d t 
    + \mathbb{E} \int_{Q} v_{\varepsilon_{k}} Z d t 
    ,
\end{align*}
and
\begin{align*}
    \mathbb{E} \int_{\mathcal{M}} \tilde{y}(T) q 
    - \mathbb{E} \int_{\mathcal{M}} y_{0} z(0)
    =
     \mathbb{E} \int_{Q} p \tilde{y} d t 
    + \mathbb{E} \int_{Q} F z d t 
    + \mathbb{E} \int_{0}^{T} \int_{\mathcal{M} \cap G_{0}} \hat{u} z d t 
    + \mathbb{E} \int_{Q} \hat{v} Z d t 
    .
\end{align*}
From this and \cref{eqLinearE13}, we get $ \hat{y} = \tilde{y} $ and $  \tilde{y}(T) = \hat{y}_{T} $ $ \mathbb{P} $-a.s.
By the weak convergence \cref{eqLinearE13}, Fatou's lemma and \cref{eqLinearE11}, we deduce \cref{eqLinearYT,eqLinearYUV}.
\end{proof}

\section{Proof of the semilinear result}

\begin{proof}[Proof of \cref{eqNullControllabilityForGeneralEquation}]

Let us consider $ f $ fulfilling assumptions (A1)-(A3). 
Recalling \cref{eqSlambdamu},  for all $ \mu, \lambda  $ and $ h $ such that \cref{eqDisCarleman} holds, define the nonlinear map
\begin{align*}
    \mathcal{N}: F \in \mathfrak{S}_{\lambda, \mu,h} \mapsto f(\hat{y}) \in \mathfrak{S}_{\lambda, \mu,h}
    ,
\end{align*}
where $ \hat{y} $ solves \cref{eqDiscreteLinear} with the initial datum $ y_{0} $ and the source term $ F $.

We fix $ \mu $ and $ m=1 $.
We can show that $ \mathcal{N} $ is well defined.
In face,  noting $\theta\geq1$, by \cref{eqSlambdamuNorm,eqWeightedFunctionTime,eqWeightFuncitonDefine,eqLinearYUV}, we have 
\begin{align*}
    | \mathcal{N}(F)|_{\mathfrak{S}_{\lambda, \mu,h}}^{2} 
    &  = 
    \mathbb{E} \int_{Q} s^{-3} e^{- 2 s \phi} |f(\hat{y})|^{2} d t
    \\
    & \leq 
    \lambda^{-3} L^{2} \mathbb{E} \int_{Q}  e^{- 2 s \phi} |\hat{y}|^{2} d t
    \\
    & \leq 
    C  \lambda^{-3} L^{2} \Big(
        \mathbb{E} \int_{\mathcal{M}} s^{-2} e^{- 2 s \phi} |y|^{2} \Big|_{t = 0}
        +  \mathbb{E} \int_{Q} s^{-3} e^{- 2 s \phi} |F|^{2} d t 
    \Big)
    \\
    & <
    \infty
    .
\end{align*}
This shows that $ \mathcal{N} $ is well defined.

Next, we choose $ \lambda$  such that $ \mathcal{N} $ is strictly contractive.
Consider two source terms $ F_{1}, F_{2} \in \mathfrak{S}_{\lambda,\mu,h} $. 
We denote the solutions of the corresponding equations by $ \hat{y}_{1} $ and $ \hat{y}_{2} $, respectively.
From \cref{eqSlambdamuNorm,eqWeightFuncitonDefine,eqLinearYUV}, we have
\begin{align*}
    | \mathcal{N}(F_{1}) - \mathcal{N}(F_{2})|_{\mathfrak{S}_{\lambda, \mu,h}}^{2} 
    & =
    \mathbb{E} \int_{Q} s^{-3} e^{- 2 s \phi} | f(\hat{y}_{1}) - f(\hat{y}_{2})|^{2} d t 
    \\
    & \leq
    L^{2} \lambda^{-3} \mathbb{E} \int_{Q}   e^{- 2 s \phi} | \hat{y}_{1} - \hat{y}_{2}|^{2} d t 
    \\
    & \leq
    C L^{2} \lambda^{-3} |  F_{1}  -  F_{2} |_{\mathfrak{S}_{\lambda, \mu ,h}}^{2}
    . 
\end{align*}
Choosing $ \lambda $ sufficiently large such that 
\begin{align}
    \label{eqL}
    C L^{2} \lambda^{-3} < 1,
\end{align}
the map $ \mathcal{N} $ is strictly contractive.

By the Banach fixed point theorem, we conclude that $ \mathcal{N} $ has a unique fixed point $ F $ in $ {\mathfrak{S}}_{\lambda,\mu,h}$.
Let $ y $ be the solution associated to this $ F $ with the control $ (u ,v) $. Then, we have $ F = f(y) $.

From \cref{eqLinearYUV}, it follows that 
\begin{align*}
    & 
    \mathbb{E} \int_{Q} e^{- 2 s \phi} |y|^{2} dt 
    + \mathbb{E} \int_{0}^{T} \int_{\mathcal{M} \cap G_{0}} s^{-3} e^{- 2 s \phi} |u|^{2} dt 
    + \mathbb{E} \int_{Q} s^{-2} e^{- 2 s \phi} |v|^{2} dt 
    \\
    & \leq 
    C   \mathbb{E} \int_{\mathcal{M}} s^{-2} e^{- 2 s \phi} |y|^{2} \Big|_{t = 0}
    + C  L^{2} \lambda^{-3} \mathbb{E} \int_{Q}   e^{- 2 s \phi} |y|^{2} d t 
    .
\end{align*}
By the above inequality and \cref{eqL}, we get 
\begin{align}
    \label{eqNonE1} \notag
    & 
    \mathbb{E} \int_{Q} e^{- 2 s \phi} |y|^{2} dt 
    + \mathbb{E} \int_{0}^{T} \int_{\mathcal{M} \cap G_{0}} s^{-3} e^{- 2 s \phi} |u|^{2} dt 
    + \mathbb{E} \int_{Q} s^{-2} e^{- 2 s \phi} |v|^{2} dt 
    \\
    & \leq 
    C   \mathbb{E} \int_{\mathcal{M}} s^{-2} e^{- 2 s \phi} |y|^{2} \Big|_{t = 0}
    ,
\end{align}
which implies  that \cref{eqEstimateOfSolutionAndControl} holds true for $ h \leq h_{0} $ and $ \lambda h (\delta T)^{-1} \leq \varepsilon_{0} $.

From \cref{eqLinearYT,eqNonE1,eqEpsilon0}, we obtain 
\begin{align*}
    \mathbb{E} \int_{\mathcal{M}} |y(T)|^{2} 
    & \leq 
    C \mathcal{E}_{\lambda,\mu,h}  \Big(
        \mathbb{E} \int_{\mathcal{M}} s^{-2} e^{- 2 s \phi} |y|^{2} \Big|_{t = 0}
        + \mathbb{E} \int_{Q} s^{-3} e^{- 2 s \phi} |F|^{2} d t 
    \Big) 
    \\
    & \leq 
    C \mathcal{E}_{\lambda,\mu,h}  \Big(
        \mathbb{E} \int_{\mathcal{M}} s^{-2} e^{- 2 s \phi} |y|^{2} \Big|_{t = 0}
        + L^{2} \mathbb{E} \int_{Q} e^{- 2 s \phi} |y|^{2} d t 
    \Big) 
    \\
    & \leq 
    C (1+L^{2})\mathcal{E}_{\lambda,\mu,h}  
    \mathbb{E} \int_{\mathcal{M}} s^{-2} e^{- 2 s \phi} |y|^{2} \Big|_{t = 0}
    \\
    & \leq 
    C (1+L^{2}) h^{-2} e^{C \lambda (2 \theta(0)-\theta(T)) }  
    \mathbb{E} \int_{\mathcal{M}}    |y_{0}|^{2}  
    .
\end{align*}
Choose $ \delta_{0} > 0$ such that for all $ \delta < \delta_{0} $, it holds that
\begin{align*}
    2\theta(0)-\theta(T) = 4 - (\delta T)^{-1} \leq - \frac{1}{2} (\delta T)^{-1}
    .
\end{align*}
Note that we need $ h \leq h_{0} $ and $ \lambda h (\delta T)^{-1} \leq \varepsilon_{0} $, where $ \lambda $ is fixed. 
Let $ h_{1} = \frac{\varepsilon_{0} \delta_{0} T}{\lambda} $.
For any $ h \leq \min \{h_{0}, h_{1}\} $, we choose $ \delta = \frac{h}{ h_{1} } \delta_{0} \leq \delta_{0} $.
Then, we obtain $ \lambda h (\delta T)^{-1} = \varepsilon_{0} $. 
By the above argument, we get 
\begin{align*}
    \mathbb{E} \int_{\mathcal{M}} |y(T)|^{2} 
    \leq C e^{-\frac{C}{h}} \mathbb{E} \int_{\mathcal{M}}    |y_{0}|^{2}  
    .
\end{align*}
\end{proof}

\appendix

\end{document}